\documentclass[12pt]{amsart}%
\usepackage{amsmath}
\usepackage{amssymb}
\usepackage{amsthm}
\usepackage{amscd}
\usepackage[dvipdfmx]{graphicx}
\usepackage[mathscr]{eucal}
\usepackage{bm}
\usepackage{xcolor}
\makeindex
\newtheorem{Thm}{Theorem}[section]
\theoremstyle{definition}
\newtheorem{Theorem}[Thm]{Theorem}
\newtheorem{Lemma}[Thm]{Lemma}
\newtheorem{Corollary}[Thm]{Corollary}
\newtheorem{Proposition}[Thm]{Proposition}

\newtheorem{Example}[Thm]{Example}
\theoremstyle{remark}
\newtheorem{Remark}{Remark}

\font\ym=msbm10  
\newcommand{\Aut}{{\rm Aut}}

\newcommand{\Z}{\text{\ym Z}}

\newcommand{\R}{\text{\ym R}}

\newcommand{\T}{\text{\ym T}}
\newcommand{\C}{\text{\ym C}}

\newcommand{\sA}{\mathscr A}

\newcommand{\sR}{\mathscr R}

\newcommand{\sT}{\mathscr T}

\newcommand{\End}{\hbox{\rm End}}

\title[boundary measure]{Boundary measures of holomorphic functions on the imaginary domain}
\author{Shigeru Yamagami}
\begin{document}
\maketitle
\begin{center}
Graduate School of Mathematics
\end{center}
\begin{center}
Nagoya University 
\end{center}
\begin{center} 
Nagoya, 464-8602, JAPAN
\end{center} 

\medskip

\medskip


\begin{abstract}
  In connection with the Herglotz-Nevanlinna integral representation of so-called Pick functions,
  we introduce the notion of boundary measure of holomorphic functions on the imaginary domain and elucidate some of basic properties.
\end{abstract}

\section*{Introduction}

In our previous paper \cite{HY}, we investigated supplementary results to the classical integral representation theorem of Herglotz, Riesz
and Nevanlinna on holomorphic endofunctions of the upper half plane.
As an integral transformation, it is closely related to the Cauchy-Stieltjes transform and we know a satisfactory characterization of
the transform of a positive measure as a holomorphic function on the upper half plane. 
Although the transform itself is applicable to complex measures, its description as a holomorphic function seems somewhat awkward.
We have however a pragmatic characterization in terms of existence of boundary measures if its support is not full (\cite{Y}),
which was one of motivations in \cite{HY}.

Here we shall extend it to boundary measures of any support under the assumption of some estimates 
on the boundary behaviour of holomorphic functions, which is in fact necessary for the existence of boundary measure.

To deal with the case of full support, we need to work on several (at least two) local coordinates to describe boundary limits.
The main problem there is the compatibility of locally defined boundary measures and we here impose localized boundary estimates
to check the compatibility of measures arising from local coordinates.
We then restrict ourselves to the most simple (and seemingly most interesting) choice of local coordinates, i.e.,
the original one $z$ and its inversion $\widetilde z = -1/z$,
and show that the associated local measures are related to each other under the inversion 
and the joined measure on the whole boundary turns out to fulfill the integral representation.

During its formulation, it turns out that the assumption on boundary behavior enables us to capture boundary limits
as distributions of mild singularity (such as Cauchy's principal value) and
the problem of coordinate transformation is soloved first within boundary distributions.
Boundary measures are then worked out as a result of relevant facts on boundary distributions. 

The author is grateful to M.~Nagisa and M.~Uchiyama for fruitful conversations, which were big impetuses to the present work.

\section{Representing Measures}
Let $\overline{\C} = \C \sqcup \{ \infty\}$ be the Riemann sphere and $\overline{\R} = \R \sqcup \{\infty\}$
be the extended real line in $\overline{\C}$. 
For a holomorphic function $\varphi(w)$ of $w \in \C \setminus \R$,
a complex Radon measure $\lambda$ on $\overline{\R}$ is called the \textbf{representing measure} of $\varphi$ if it satisfies 
\[
  \varphi(w) = \frac{\varphi(i) + \varphi(-i)}{2} + \int_{\overline{\R}} \frac{1+sw}{s-w}\, \lambda(ds)
  \quad
  (w \in \C \setminus \R). 
\]
The representing measure, if any exists, is unique
because $\{(1+sw)/(s-w); w \in \C \setminus \R\}$ is total in the space $C(\overline{\R})$ of continuous functions on $\overline{\R}$.
Related to the representing measure, given a holomorphic function $\varphi$ on $\C \setminus (\R \sqcup K)$
with $K$ a compact subset of $\C \setminus \R$,
the \textbf{boundary support} $[\varphi]$ of $\varphi$ is defined to be the complement in $\overline{\R}$ of an open subset 
\[
\{ c \in \overline{\R}; \text{$\varphi(z)$ is continuously extended to a neighborhood of $c$ in $\overline{\C}$} \}. 
\]
Note that $[\varphi]$ is a support in the sense of hyperfunction, which turns out to be the support of the representing measure.

The existence of representing measure is a strong condition on holomorphic functions.

\begin{Example}
  A holomorphic function $\varphi$ on $\C \setminus \R$ is extended to to an entire function on $\C$ if and only if $[\varphi] = \{ \infty\}$. 
  As for an entire function, the representing measure exists if and only if it is an affine function. 
\end{Example}

By the holomorphic change-of-variable
\[ 
  z = \frac{i-w}{i+w}, 
\quad 
  w = i\frac{1-z}{1+z} 
\]
of $\overline{\C}$,
the upper half plane $\C_+ = \{ w \in \C; \text{Im}\,w > 0\}$ is bijectively mapped onto the unit disk $D = \{ z \in \C; |z| < 1\}$
and, under the correspondence $\phi(z) = -i\varphi(w)$, 
we have the disk version of representing measure:
Let $\T = \{ z \in \C; |z| = 1\}$ be the boundary of $D$. A holomorphic function $\phi$ on $\overline{\C} \setminus \T$ 
admits the representing measure if there exists a complex Radon measure $\mu$ on $\T$ satisfying
\[
  \phi(z) = \frac{\phi(0) + \phi(\infty)}{2} + \frac{1}{2\pi} \int_\T \frac{\zeta + z}{\zeta - z}\, \mu(d\zeta).
\]

In view of the Jordan decomposition on signed measures and the Herglotz-Riesz integral representation,
functions having representing measures are exactly linear combinations of positive ones (so-called Pick functions on $\C_+$).

To make the relevant structure clear, we introduce a *-operation on the space $\sR$ of holomorphic functions on $\C \setminus \R$ by
$\varphi^*(w) = \overline{\varphi(\overline{w})}$ and a *-operation\footnote{A minus sign is put here to relate the real part of $\phi$ to
  the imaginary part of $\varphi$.} on the space $\sT$ of holomorphic functions
on $\overline{\C} \setminus \T$ by $\phi^*(z) = - \overline{\phi(1/\overline{z})}$.
Then $\sR$ and $\sT$ are *-isomorphic by the correspondence $\varphi(w) = i \phi(z)$ so that
reality $\overline{\varphi(w)} = \varphi(\overline{w})$ ($w \in \C \setminus \R$) is equivalent to $\varphi^* = \varphi$ for $\varphi \in \sR$,
while imaginarity $\overline{\phi(z)} = - \phi(1/\overline{z})$ ($z \in \overline{\C} \setminus \T$) is equivalent to
$\phi^* = \phi$ for $\phi \in \sT$.

The *-operations behave well with respect to (the existence of) representing measures:
In the case of the extended real boundary $\overline{\R}$,
the representing measure of $\varphi^*$ is given by the complex conjugate $\overline{\lambda}$ of $\lambda$ as
\[
  \varphi^*(w) = \frac{\varphi^*(i) + \varphi^*(-i)}{2} + \int_{\overline{\R}} \frac{1+sw}{s-w}\, \overline{\lambda}(ds), 
\]
and, in the case of the circle boundary $\T$, as 
\[
  \phi^*(z) = \frac{\phi^*(0) + \phi^*(\infty)}{2} + \frac{1}{2\pi} \int_\T \frac{\zeta + z}{\zeta - z}\, \overline{\mu}(d\zeta). 
\]
Thus the existence of representing measure is reduced to the case of signed measures
by extracting $(\varphi + \varphi^*)/2$ and $(\varphi - \varphi^*)/2i$ (resp.~$(\phi + \phi^*)/2$ and $(\phi - \phi^*)/2i$)
with the corresponding signed measures $(\lambda + \overline{\lambda})/2$ and $(\lambda - \overline{\lambda})/2i$
(resp.~$(\mu + \overline{\mu})/2$ and $(\mu - \overline{\mu})/2i$). 

Recall here that the Poisson integral formula of disk version 
\[
\phi(z) = i \text{Im}\, \phi(0) + \frac{1}{2\pi} \oint_0^{2\pi} \frac{e^{it} + z}{e^{it} - z}\, \text{Re}\, \phi(e^{it})\, dt 
\]
holds for $|z| < 1$ if a holomorphic function $\phi$ on $D$ is
continuously extended to $\overline{D} = D \sqcup \T$.
Then the imaginary extension of $\phi$, i.e., $\phi(z) = - \overline{\phi(1/\overline{z})}$ for $z \in \overline{\C} \setminus \overline{D}$,  
satisfies $\phi^*(z) = \phi(z)$ ($z \in \overline{\C} \setminus \T$) and
\[
  i\text{Im}\,\phi(0) = \frac{\phi(0) - \overline{\phi(0)}}{2} = \frac{\phi(0) + \phi^*(\infty)}{2}
  = \frac{\phi(0) + \phi(\infty)}{2}. 
\]

Moreover, if $z \in \overline{\C} \setminus \overline{D}$, we have 
\[
  \phi(z) = \phi^*(z) = - \overline{\phi(1/\overline{z})}
  = i \text{Im}\, \phi(0) + \frac{1}{2\pi} \oint_0^{2\pi} \frac{e^{it} + z}{e^{it} - z}\, \text{Re}\, \phi(e^{it})\, dt.  
\]
Thus, for the imaginary extension of $\phi$ in this case, the representing measure exists and takes the form
\[
  \displaystyle \mu(dt) = \text{Re}\,\phi(e^{it})\,dt = \lim_{r \uparrow 1} \text{Re}\, \phi(re^{it})\, dt
\]
relative to the parametrization $\zeta = e^{it}$ of $\T$. 
Notice here that, the boundary values exist both for $\lim_{r \uparrow 1}\phi(re^{i\theta})$ and $\lim_{r \downarrow 1} \phi(re^{i\theta})$ but
their difference amounts to 
\[
    \lim_{r \uparrow 1} \phi(r e^{i\theta}) -  \lim_{r \downarrow 1} \phi(r e^{i\theta})
    =  \lim_{r \uparrow 1} \phi(r e^{i\theta}) + \lim_{r \downarrow 1} \overline{\phi(e^{i\theta}/r)}
    = 2 \text{Re}\,\phi(e^{i\theta}).
\]
  
Returning to the initial situation without boundary continuity,
for the existence problem of representing measure $\mu$ of $\phi = \phi^*$, we may further assume that
$\text{Im}\, \phi(0) = 0 \iff \phi(0) + \phi(\infty) = 0$ without modifying $\mu$ 
by adjusting imaginary additive constants.


In that case (i.e., $\phi = \phi^*$ and $\text{Im}\,\phi(0) = 0$), if a representing measure $\mu$ exists, 
its Jordan decomposition $\mu = \mu_+ - \mu_-$ with $|\mu| = \mu_+ + \mu_-$ and $\mu_+ \perp \mu_-$ 
gives rise to the decomposition $\phi = \phi_+ - \phi_-$, where
\[
    \phi_\pm(z) = \frac{1}{2\pi} \oint_0^{2\pi} \frac{\zeta + z}{\zeta - z}\, \mu_\pm(dt)
\]
are holomorphic functions of $z \in \overline{\C} \setminus \T$ satisfying $\phi_\pm^* = \phi_\pm$,
$\text{Im}\,\phi_\pm(0) = 0$ and $\text{Re}\, \phi(z) \geq 0$ ($|z| < 1$)
(i.e., $\text{Re}\, \phi(z) \leq 0$ for $|z| > 1$).
Moreover, the representing measure $\mu_\pm$ of $\phi_\pm$ is expressed by
\[
    \oint f(t)\, \mu_\pm(dt) = \lim_{r \uparrow 1} \oint f(t) \text{Re}\, \phi_\pm(re^{it})\, dt
\]
for each $f \in C(\T)$ (Herglotz-Riesz-Nevanlinna) and hence the signed measure on $\T$ 
\[
    \mu_r(dt) = \Bigl(\text{Re}\, \phi_+(re^{it}) - \text{Re}\, \phi_-(re^{it})\Bigr)\, dt
    = \frac{1}{2} \Bigl(\phi(re^{it}) - \phi(r^{-1}e^{it})\Bigr)\, dt 
\]
parametrized by $0 < r < 1$ 
converges weakly to $\mu$ as $r \uparrow 1$.



 \section{Boundary Limits by Circles}
 Now reverse the process and suppose that the signed measure $\mu_r$ on $\T$ associated with $\phi = \phi^*$ converges
 as a linear functional on $C(\T)$ to a signed Radon measure $\mu$ when $r \uparrow 1$.
 
 Then $\mu$ is the representing measure of $\phi$; 
 letting $\phi_\mu = \phi_+ - \phi_-$, we need to show $\phi = \phi_\mu$.
 Since $2\mu(dt) = \lim_{r \uparrow 1} \bigl(\phi_\mu(re^{it}) - \phi_\mu(r^{-1} e^{it})\bigr)\, dt$ as observed above,
 replacing $\phi$ with $\phi - \phi_\mu$, we may assume that $\mu = 0$ and the problem is to show that $\phi = 0$
 under the extra assumption $\text{Im}\,\phi(0) = 0$.

 Though this can be deduced easily from analogous facts on harmonic functions (see \cite[Theorem 1.1]{Du} for example), 
 we shall here present a direct proof by modifying arguments in \cite[Appendix A]{Y}.
 
 Regard $\mu_r$ as a bounded linear functional on $C(\T)$ so that the assumption $\mu = 0$ implies
 $\lim_{r \uparrow 1} \mu_r(f) = 0$ for $f \in C(\T)$. Since $\mu_r(f)$ is continuous in $0 < r < 1$, this implies that
 $\mu_r(f)$ is extended to a continuous function of $0 < r \leq 1$ so that $\mu_1(f) = 0$.
 By the Banach-Steinhaus theorem, the functional norm $\| \mu_r\|$ is then bounded in $0 < r < 1$. 
 Notice also that
 \[
   \lim_{r \downarrow 0} \phi(re^{it}) = \phi(0),
   \quad
   \lim_{r \uparrow \infty} \phi(re^{it})
   = - \lim_{r \uparrow \infty} \overline{\phi(r^{-1} e^{it})} = - \overline{\phi(0)}. 
 \]
 
 Given $f \in C(\T)$, the rotational average (convolution) defined by a periodic integral 
 \[
   \phi_f(z) = \oint_0^{2\pi} f(t) \phi(ze^{it})\, dt
 \]
 is a holomorphic function of $z \in \overline{\C} \setminus \T$ satisfying $\phi_f^* = \phi_{\overline{f}}$. 
 By using translation $(T_\theta f)(t) = f(t-\theta)$ and the polar form $z = r e^{i\theta}$, we have 
 \begin{align*}
   \mu_r(T_\theta f)
   &= \oint f(t-\theta) \Bigl( \phi(re^{it}) - \phi(r^{-1} e^{it}) \Bigr)\,dt\\
   &= \oint f(t) \Bigl( \phi(r e^{i\theta}e^{it}) - \phi(r^{-1} e^{i\theta} e^{it}) \Bigr)\,dt\\
   &= \phi_f(z) - \phi_f(1/\overline{z}). 
 \end{align*}
 Since $T_\theta f$ is norm-continuous in $\theta$ and $\phi_f(z)$ is continuous in $z \in  \overline{\C} \setminus \T$,
 the norm boundedness of $\mu_r$ ($0 < r < 1$) implies that $\mu_r(T_\theta f)$ is extended to a continuous function of $r \in (0,1]$
 and $\theta \in \R$ in such a way that $\mu_1(T_\theta f) = 0$ ($\theta \in \R$). 

 Thus the harmonic function $\phi_f(z) - \phi_f(1/\overline{z})$ of $z \in \C \setminus \T$ is continuously extended to $\overline{\C}$
 so that it vanishes on $\T$, whence it vanishes identically by the maximum principle on its real and imaginary parts.

 By choosing approximate delta functions for $f$ and taking its limit,
 we have $\phi(z) = \phi(1/\overline{z})$, which together with $\phi = \phi^*$ implies that $\phi(z)$ is an imaginary constant,
 i.e., $\phi(z) = \text{Im}\,f(0) = 0$. 
 
 In this way the representing measure of $\phi$ is expressed as the limit of an approximating measure $\text{Re}\, \phi(re^{it})\, dt$
 associated with an inner circle $|z| = r < 1$ in $D$ (the inversion formula of disk version).
 In other words, the representing measure of $\phi = \phi^*$ exists if and only if
 the functional limit $\lim_{r \uparrow 1} \text{Re} \phi(r e^{it})\, dt$ exists as a signed Radon measure on $C(\T)$ with respect to
 the angle parameter $e^{it}$ of $\T$.

 By the decomposition $\phi = (\phi + \phi^*)/2 + i(\phi - \phi^*)/2i$, the following is now in our hands. 

 \begin{Theorem}\label{circle}
 A holomorphic function $\phi$ on $\overline{\C} \setminus \T$
 admits a representing measure on $\T$ if and only if
 \[
   \lim_{r \uparrow 1} \frac{1}{2} (\phi(r e^{it}) - \phi(e^{it}/r))\, dt
 \]
 exists as a weak* limit of complex Radon measures on $C(\T)$.

 Moreover, if this is the case, the representing measure of $\phi$ is given by the above limit. 
\end{Theorem}

\section{Boundary Limits by  Lines}
We next examine the approximating process in terms of the holomorphic function $\varphi(w)$ on the imaginary domain $\C \setminus \R$
through a M\"obius transform $w = i(1-z)/(1+z)$. 
First recall that the inversion formula for the representing measure $\lambda$ of $\varphi$ takes
the following form\footnote{This is not just a rewriting of Theorem~\ref{circle} but an immediate consequence of the classical case.
For the atomic mass formula, see also \cite[Appendix A]{HY}.}:
\[
  \lambda(du) = \lim_{v \to +0} \frac{\varphi(u+iv) - \varphi(u-iv)}{2\pi i(1+u^2)}\, du
\]
as a complex Radon measure on $\R$, i.e.,
\[
  \int_\R f(u)\, \lambda(du) = \lim_{v \to +0} \int_{-\infty}^\infty f(u) \frac{\varphi(u+iv) - \varphi(u-iv)}{2\pi i(1+u^2)}\, du
\]
for each continuous function $f$ on $\overline{\R}$ (see Appendix~C) and 
\[
\pm i \lambda(\{\infty\}) = \lim_{v \to +\infty} \frac{\varphi(\pm iv)}{v}.  
\]
Furthermore, related to atomic parts of $\lambda$ in $\R$, we notice that 
the holomorphic function $\varphi$ having a representing measure satisfies Vladimirov's estimate (Theorem~\ref{Vl}) globally on $\C \setminus \R$.

To rephrase this property in a local fashion, we introduce several terminology.
By a \textbf{boundary neighborhood} of $s \in \R$, we shall mean $U \setminus \R$ with $U$ an open neighborhood of $s$ in $\C$.
Boundary neighborhoods are further localized and $U\cap \C_\pm$ is called a boundary neghborhood of $s\pm i0$.
Here $\C_\pm$ denotes the upper/lower half plane. 
We say that a holomorphic function defined on a boundary neighborhood $U \cap \C_\pm$ of $s \pm i0$
\textbf{behaves simply}\footnote{Our usage of `simple' here is in accordance with that of zeros and poles.} at $s\pm i0$
if $(\text{Im}\,z) \varphi(z)$ is bounded as a function of $z \in U' \cap \C_\pm$ for a small neighborhood $U'$ of $s$.

When $\varphi$ is defined on a boundary neighborhood of $s$ and $\varphi$ behaves simply at both of $s\pm i0$,
we say that $\varphi$ behaves simply at $s$.
For a holomorphic function $\varphi$ defined on $\C_\pm \setminus K$ with $K$ a compact subset of $\C \setminus \R$,
$\varphi$ is said to behave simply at $\infty \pm i0$ if the inversion
$\widetilde\varphi(z) \equiv \varphi(-1/z)$ of $\varphi$ behaves simply at $z = 0 \pm i0$, i.e., if
$\varphi(z) \text{Im}\,z/|z|^2$ is bounded for a sufficiently large $z \in \C_\pm$.

For a subset $S \subset \overline{\R}$, a boundary neighborhood of $S \pm i0$ (resp.~$S$) is defined to be $U \cap \C_\pm$
(resp.~$U \setminus \overline{\R}$) with $U$ an open neighborhood of $S$ in $\overline{\C}$ and
a holomorphic function $\varphi$ defined on a boundary neighborhood of $S\pm i0$ (resp.~$S$)
is said to behave simply on $S\pm i0$ (resp.~$S$)
if $\varphi$ behaves simply at $s \pm i0$ (resp.~$s$) for each $s \in S$.
By a compactness argument, a holomorphic function $\varphi$ defined on a boundary neighborhood of $S\pm i0$ (resp.~$S$) behaves simply on
$S \pm i0$ (resp.~$S$) if and only if
\[
  \frac{\text{Im}\, z}{1 + |z|^2} \varphi(z)
\]
is bounded on $U'\cap C_\pm$ (resp.~$U' \setminus{\R}$) for a small open neighborhood $U'$ of $S$ in $\overline{\C}$. 

By Appendix B, if $S$ is an open subset of $\R$ and $\varphi$ behaves simply on $S\pm i0$, the limit
(called the  \textbf{boundary limit} of $\varphi$ on $S$)
\[
  \varphi(x\pm i0) = \lim_{y \to +0} \varphi(x \pm iy)
\]
exists as a Schwartz distribution in $S$, i.e., for any $f \in C_c^\infty(S)$,
\[
  \lim_{y \to +0} \int_S f(x) \varphi(x \pm iy)\, dx
\]
exists. 


Distributional boundary limits can be also considered on the circle boundary.
This is simply achieved by using a periodic function $e^{iz}$ of period $2\pi$ as a (reverse) coordinate of $\C^\times$: In terms of
the polar expression $re^{it}$ of $e^{iz}$, $\text{Im}\,z = -\log r$ and the condition $\text{Im}\,z \to \pm 0$ is
equivalent to $r \to 1 \mp 0$ in such a way that $-(\log r)/(1-r) \to 1$ as $r \to 1$. 
Thus, by the polar coordinates $(r,t)$ of $e^{iz} \in \C^\times$ around $e^{i\theta} \in \T$, we see that 
$\phi(e^{iz})$ behaves simply at $z=\theta+i0$ (resp.~$z = \theta-i0$) if and only if
$(1-r) \phi(re^{it})$ is bounded on $\theta-\delta \leq t \leq \theta + \delta$ and $e^{-\delta} \leq r < 1$ (resp.~$1 < r \leq e^\delta$)
for some $\delta>0$. Moreover, if this is the case,
\[
  \lim_{r \updownarrow 1} \phi(re^{it}) = \lim_{\tau\updownarrow 0} \phi(e^{it - \tau})
  \equiv \phi(e^{it \mp 0})
\]
exists as a distribution of $t$ near the point $\theta$.

Returning to the inversion formula of representing measures, 
the big difference between these two versions in approximation is that the convergence at $\infty$ is dealt with separately for $\varphi$,
whereas it is overall for $\phi$.
This is because the approximation for $\varphi$ is based on lines $\R + iv$ ($v > 0$)
in $\C_+$ which constitute circles tangential to the boundary
$\overline{\R}$, whereas an inner circle in the disk version never touches the boundary $\T$.
Thus the approximating processes themselves are different in two boundaries and it is not clear at all whether two approximations 
give rise to the same measure or not, even if their limits are assumed to exist.
In fact, a direct linkage between limit measures for different boundaries
is difficult to estabslish and we shall circumvent it by relating boundary distributions instead of representing measures. 


\section{Boundary Distributions}
To begin with, let the inner circle $z = r e^{it}$ ($-\pi \leq t \leq \pi$ with $\pm \pi$ identified
in accordance with $e^{\pm i\pi} = -1$) be M\"obius-transformed to a circle $C_r$ in $\C_+$ by
 \[
   w_r(t) = i \frac{1-r e^{it}}{1 + r e^{it}}
 \]
 so that
 \[
   \frac{1-r}{1+r} \leq \text{Im}\, w_r(t) \leq \frac{1+r}{1-r}
 \]
 with extremal values attained at $w_r = i (\frac{1-r}{1+r})^{\pm 1}$. 
 Thus the center of $C_r$ is located at
 \[
   \frac{i}{2} \left(\frac{1-r}{1+r} + \frac{1+r}{1-r}\right)
   = i \frac{1+r^2}{1-r^2}
 \]
 with its radius equal to 
 \[
   \frac{1}{2} \left( \frac{1+r}{1-r} - \frac{1-r}{1+r} \right) = \frac{2r}{1-r^2}. 
 \]

 In terms of a new variable $s = \tan(t/2) = w_1(t) = i(1-e^{it})/(1+e^{it})$ parametrizing $\R \subset \C$,
 we have
  \[
    \phi(re^{it}) = -i \varphi\left(\frac{(1+r)s + i(1-r)}{1+r - i(1-r)s}\right),
    \quad
 dt = \frac{2}{1+s^2} ds.
 \]
Remark that $e^{it} = (1+is)/(1-is)$ and
 \[
   s = \frac{(1+r)w_r(t) - i(1-r)}{1+r + i(1-r)w_r(t)}
   \iff
   w_r(t) = \frac{(1+r)s + i(1-r)}{1+r - i(1-r)s} \equiv s_r. 
 \]
 Here the notation $s_r$ indicates that it is an $r$-deformation of the real variable $s$ into complex one and
 $s_r$ ($s \in \R$) traces the circle $C_r$ so that $0_r = i(1-r)/(1+r)$ and $(\pm\infty)_r = i(1+r)/(1-r)$. 
 Denote the real and imaginary parts of $s_r$ by $u_r(s)$ and $v_r(s)$ respectively, which satisfy
\[
  \frac{u_r(s)}{s} = \frac{4r}{(1+r)^2 + (1-r)^2s^2} < 1,
  \quad 
 v_r(s) = \frac{(1-r^2)(1+s^2)}{(1+r)^2 + (1-r)^2s^2}
\]
and, by circle geometry, $v_r(s)$ is monotone increasing on $s \geq 0$,
whereas $u_r(s)$ is monotone increasing on a closed interval $[-a,a]$ ($a\geq 0$) if and only if $a \leq 2r/(1-r^2)$.

\begin{Lemma}\label{simples}
  Let $\varphi$ be a holomorphic function defined on a boundary neighborhood of $s_0 \pm i0$ ($s_0 \in \overline{\R}$), 
  $\phi(z) = -i\varphi(w)$ be the M\"obius transform of $\varphi$ and set $t_0 = 2\arctan s_0 \in (-\pi,\pi]$.
  
  Then $\phi(e^{iz})$ behaves simply at $z = t_0 \pm i0$ if and only if $\varphi(w)$ behaves simply at $w = s_0 \pm i0$.  
\end{Lemma}

\begin{proof}
  In the expression
  \[
    (1-r) \phi(re^{it}) = -i \frac{1-r}{v_r(s)} (1 + |s_r|^2) \frac{\text{Im}\,s_r}{1+|s_r|^2} \varphi(s_r),
  \]
  the identity
  \[
   1 + |s_r|^2 = 1 + \frac{(1+r)^2s^2 + (1-r)^2}{(1+r)^2 + (1-r)^2s^2} = \frac{2(1+r^2)(1+s^2)}{(1+r)^2 + (1-r)^2s^2}
  \]
  is used to see that
  \[
    \frac{1-r}{v_r(s)} (1 + |s_r|^2) = 2 \frac{1+r^2}{1+r}
  \]
  and its reciprocal are bounded if so is $r>0$. 
\end{proof}

\begin{Corollary}
  A holomorphic function $\varphi$ defined on a boundary neighborhood of $\overline{\R}$ behaves simply on $\overline{\R}$ if and only if
  $\phi(e^{iz})$ behaves simply on $[-\pi,\pi]$. Moreover, if this is the case, 
  boundary limits $\phi(e^{it \pm 0})$ exist as periodic distributions of period $2\pi$ on $\R$, i.e., as distributions on $\T$. 
\end{Corollary}

In view of the above corollary, naturally raised is the question of how limit distributions $\varphi(x \pm i0)$, $\widetilde\varphi(x \pm i0)$
on $\R$ and $\phi(e^{it \mp 0})$ on $\T$ are related, which shall be answered in what follows
under the assumption that $\varphi$ behaves simply on $\overline{\R} \pm i0$.

To make the comparison explicit, choose a closed interval $[-a,a]$ ($a>0$) in $\R$ for the $s$-variable and
consider a linear functional of $f \in C_0(-a,a)$ 
($C_0$ stands for the set of continuous functions vanishing at infinity or on the boundary)
defined with the choice $\alpha = 2 \arctan(a)$ by 
\[
 \int_{-\alpha}^\alpha f(\tan\frac{t}{2}) \phi(r e^{it})\, dt = - i\int_{-a}^a f(s) \varphi(s_r) \frac{2}{1+s^2}\, ds, 
\]
whose limit for $r \uparrow 1$ should be compared with the line integral limit 
\[
 - i  \lim_{y \downarrow 0} \int_{-a}^a f(s) \varphi(s+iy) \frac{2}{1+s^2}\, ds. 
\]

More precisely, under the simple behavior of relevant functions, 
we try to show that these limits coincide.
To dispose of this problem, we shall link $r$ with $y$ so that $r \uparrow 1$ is equivalent to $y \downarrow 0$ and 
\[
  \lim \int_{-a}^a f(s) \Bigl( \varphi(s_r) - \varphi(s+iy)\Bigr)\, \frac{ds}{1+s^2} = 0. 
\]

To this end, we further restrict $f$ to be analytic on $[-a,a]$, i.e., $f$ is extended to an analytic function in a neighborhood of $[-a,a]$
so that $f(\pm a) = 0$. 
Since $s_r$ converges to $s$ for the limit $r \uparrow 1$, 
\[
 f_r(w) = f\left(\frac{(1+r)w - i(1-r)}{1+r + i(1-r)w}\right)
\]
is an analytic function of $w$ in a neighborhood of $[-a,a]$ for $r \approx 1$
in such a way that $f_r(s_r) = f(s)$ for $s \in [-a,a]$. 

Then we have an expression
\[
  \int_{-a}^a f(s) \varphi(s_r)\, \frac{ds}{1+s^2} = \int_{A_r} f_r(w) \varphi(w) \frac{dw}{1+w^2}, 
\]
where $A_r$ denotes the arc $s_r = u_r(s) + i v_r(s)$ along the circle $C_r$ from $s=-a$ to $s=a$, i.e., starting from
\[
(-a)_r = \frac{-(1+r)a + i(1-r)}{1+r + i(1-r)a} = \frac{-4ar + i(1+a^2)(1-r^2)}{(1+r)^2 + a^2(1-r)^2}
\]
and ending at 
\[
(a)_r = \frac{(1+r)a + i(1-r)}{1+r - i(1-r)a} = \frac{4ar + i(1+a^2)(1-r^2)}{(1+r)^2 + a^2(1-r)^2}.
\]

\begin{figure}[h]
  \centering
 \includegraphics[width=0.4\textwidth]{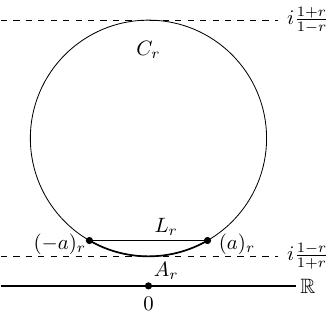}
 \caption{Arc to Line}
\end{figure}

Let $L_r$ be the line segment drawn from $(-a)_r$ to $(a)_r$. By the Cauchy integral theorem, we then have
\begin{align*}
 \int_{-a}^a f(s) \varphi(s_r)\, &\frac{ds}{1+s^2}
  = \int_{A_r} f_r(w) \varphi(w) \frac{dw}{1+w^2}
    = \int_{L_r} f_r(w) \varphi(w) \frac{dw}{1+w^2}\\
  &= \int_{-u_r(a)}^{u_r(a)} f_r(x + iv_r(a)) \varphi(x + iv_r(a)) \frac{dx}{1+(x+iv_r(a))^2}, 
\end{align*}
which suggests the choice $y = v_r(a)$ to link $r$ with $y$.

We shall derive an explicit formula of the limit 
\[
 \lim_{r \uparrow 1} \int_{L_r} f_r(w) \varphi(w) \frac{dw}{1+w^2}
\]
in a series of lemmas. 

To simplify the notation, introduce
\[
  g_r(x) = \frac{f_r(x + iv_r(a))}{ 1 + (x + iv_r(a))^2}, 
\]
which is an analytic function of $x$ in a neighborhood of $[-a,a]$ in $\R$ with an analytic parameter $r \approx 1$
satisfying 
\[
  \lim_{r \to 1} g_r(x) = \frac{f(x)}{1+x^2} \equiv g(x) 
\]
uniformly on $x$ in a neighborhood of $[-a,a]$.

\begin{Lemma} We have $u_r(a) \uparrow a$ as $r \uparrow 1$ and 
  \[
    \lim_{r \uparrow 1} \int_{[-a,a] \setminus (-u_r(a),u_r(a))} g_r(x) \varphi(x + iv_r(a))\, dx = 0.  
  \]
\end{Lemma}

\begin{proof}
  By uniform convergence of $g_r$ to $g$, $|g_r(x)| \leq G$ for $x \in [-a,a]$ and $r \approx 1$, and 
  the simple behavior of $\varphi$ at $a$ means $v_r(a) |\varphi(x + iv_r(a))| \leq C$ for $x \in [-a,a]$ and $r \approx 1$, whence 
  \[
    \int_{u_r(a)}^a  |g_r(x) \varphi(x + iv_r(a)|\, dx
    \leq C  \int_{u_r(a)}^a  \frac{|g_r(x)|}{v_r(a)} \, dx
    \leq CG  \frac{a - u_r(a)}{v_r(a)} 
  \]
  and the assertion follows from
  \[
    a - u_r(a) = \frac{a(1+a^2)(1-r)^2}{(1+r)^2 + a^2(1-r)^2}
  \]
  and then 
  \[
    \frac{a - u_r(a)}{v_r(a)} = a\frac{1-r}{1+r}. 
  \]
\end{proof}

Thus our target is reduced to the integral
\[
  \int_{-a}^a g_r(x) \varphi(x + iv_r(a))\, dx.
\]
For further reduction of $g_r$ to $g$, we use the fundamental formula of calculus of second order in the form 
\[
  g_r(x) = g(x) + (r-1) \left.\frac{\partial g_r(x)}{\partial r}\right|_{r=1}
  + \int_1^r (r-t) \frac{\partial^2 g_t(x)}{\partial t^2}\, dt. 
\]

\begin{Lemma}
  \[
  \lim_{r \uparrow 1} \int_{-a}^a \int_1^r (r-t) \frac{\partial^2 g_t(x)}{\partial t^2}\, dt \varphi(x + iv_r(a))\, dx = 0. 
  \]
\end{Lemma}

\begin{proof}
Since $\displaystyle \frac{\partial^2 g_t(x)}{\partial t^2}$ is bounded for $x \in [-a,a]$ as $t \to 1$, we have 
  \[
    \int_1^r (r-t) \frac{\partial^2 g_t(x)}{\partial t^2}\, dt = O((r-1)^2)
    \quad
    \text{for $r \approx 1$}
  \]
  and hence
  \[
    \lim_{r \uparrow 1} \frac{1}{v_r(a)} \int_1^r (r-t) \frac{\partial^2 g_t(x)}{\partial t^2}\, dt = 0
  \]
  uniformly on $x \in [-a,a]$, which is combined with the uniform boundedness of $v_r(a) \varphi(x+iv_r(a))$
  to obtain the assertion. 
\end{proof}

The whole reduction is now boiled down to the boundedness of
\[
  \int_{-a}^a \left.\frac{\partial g_r(x)}{\partial r}\right|_{r=1} \varphi(x + iv_r(a))\, dx
\]
as $r \uparrow 1$. Here the test function to be coupled with $\varphi(x+ i v_r(a))$ is given by 

\begin{Lemma}\label{first}
  \[
    \left.\frac{\partial g_r(x)}{\partial r}\right|_{r=1} = i (1+a^2) \frac{xf(x)}{(1+x^2)^2}
    - i \frac{a^2-x^2}{2(1+x^2)} f'(x). 
  \]
\end{Lemma}

\begin{proof}
  This is just by a straightforward compution. 
  Here are key steps in the process: 
  \begin{align*}
    v_r(a) &= \frac{(1-r^2)(1+a^2)}{(1+r)^2 + (1-r)^2 a^2} = \frac{(1-r^2)(1+a^2)}{(1+r)^2} + O((r-1)^2)\\
           &= (1+a^2) \frac{1-r}{1+r} + O((r-1)^2)\\
           &= \frac{1+a^2}{2} (1-r) (1 + \frac{1-r}{2}) + O((r-1)^2)\\
    &= \frac{1+a^2}{2} (1-r) + O((r-1)^2). 
  \end{align*}
  \begin{align*}
    \frac{1}{1 + (x + iv_r(a))^2}
    &= \frac{1}{1 + (x + i(1+a^2)(1-r)/2)^2} + O((r-1)^2)\\
    &= \frac{1}{1 + x^2 + i(1+a^2)(1-r)x} + O((r-1)^2)\\
    &= \frac{1}{1+x^2} \left( 1 + i(1+a^2)(r-1)\frac{x}{1+x^2} \right) + O((r-1)^2). 
  \end{align*}
  \begin{align*}
    (1+r)&(x + iv_r(a)) - i(1-r)\\
    &= (2 - (1-r))(x + i(1+a^2)(1-r)/2) - i(1-r) + O((r-1)^2)\\
    &= 2x - (1-r)x + i(1+a^2)(1-r) - i(1-r) + O((r-1)^2)\\
    &= 2x - (1-r)x + ia^2(1-r) + O((r-1)^2). 
  \end{align*}
  \begin{align*}
    \frac{1}{1+r + i(1-r)(x + iv_r(a))}
    &= \frac{1}{2 - (1-r) + i(1-r)x} + O((r-1)^2)\\
    &= \frac{1}{2} \left( 1 + \frac{1-r}{2} - i \frac{1-r}{2} x \right) + O((r-1)^2).
  \end{align*}
  \begin{align*}
    &\frac{(1+r)(x + iv_r(a)) - i(1-r)}{1+r + i(1-r)(x + iv_r(a))}\\
    &\qquad= \left( x - \frac{1-r}{2}x + ia^2\frac{1-r}{2} \right) \left( 1 + \frac{1-r}{2} - i \frac{1-r}{2} x \right) + O((r-1)^2)\\
    &\qquad= x + i \frac{1-r}{2} (a^2-x^2) + O((r-1)^2). 
  \end{align*}
  \begin{align*}
    f&\left( \frac{(1+r)(x + iv_r(a)) - i(1-r)}{1+r + i(1-r)(x + iv_r(a))} \right)\\
    &\qquad= f\left( (x + i(1-r) \frac{a^2-x^2}{2} \right) + O((r-1)^2)\\
    &\qquad= f(x) + i(1-r) f'(x) \frac{a^2-x^2}{2} + O((r-1)^2). 
  \end{align*}
  Putting all these together, $g_r(x)$ is up to $O((r-1)^2)$ given by 
  \begin{align*}
  \frac{1}{1+x^2} &\left( 1 - i(1+a^2)(1-r)\frac{x}{1+x^2} \right)
             \left( f(x) + i(1-r) f'(x) \frac{a^2-x^2}{2} \right)\\ 
    &= \frac{f(x)}{1+x^2} + i (1+a^2)(r-1) \frac{xf(x)}{(1+x^2)^2}- i (r-1) f'(x) \frac{a^2-x^2}{2(1+x^2)}. 
  \end{align*}
\end{proof}

To show the boundedness of the integral before the lemma, we need one more step to improve the results on boundary limits in Appendix~B:

Given a closed interval $[a,b]$ in $\R$, consider a continuous function $H \in C[a,b]$ and indefinitely integrate $H$ twice to have
a $C^2$ function $h$ on $[a,b]$ satisfying $h'' = H$.
Since $h$ is determined by $H$ up to a linear function, we normalize $h$ by requiring $h(a) = h(b) = 0$.

Denote by $C_0^2[a,b]$ the set of normalized $h$ for various $H \in C[a,b]$. Clearly $C_0^2[a,b]$ is a linear subspace
of $C_0(a,b)$ and isomorphic to $C[a,b]$ by the correspondence $h \mapsto h''$.

To obtain an explicit description of the normalized $h$ and $h'$. 
consider an intermediate expression $h'(x) = \int_a^x h''(t) - c$ and determine $c$ by $\int_a^b h'(x)\, dx = 0$:
The result is 
\begin{align*}
  h'(x)
  &= \int_a^x h''(t)\, dt - \frac{1}{b-a} \int_a^b (b-t) h''(t)\, dt,\\
  h(x) &= \int_a^x h'(t)\, dt = \int_a^x (x-t)h''(t)\,dt - \frac{x-a}{b-a} \int_a^b (b-t) h''(t)\, dt 
\end{align*}
with norm estimates $  \| h\|_{[a,b]} \leq (b-a) \| h'\|_{[a,b]}$ and 
\begin{align*}
  \| h'\|_{[a,b]}
  &\leq (b-a) \| h''\|_{[a,b]} + \frac{1}{b-a} \| h''\|_{[a,b]} \int_a^b (b-t)\, dt\\
  &= \frac{3}{2} (b-a) \| h''\|_{[a,b]}.
\end{align*}
Here $\| \cdot\|_{[a,b]}$ denotes the uniform norm on $[a,b]$. 

\begin{Remark}
  A similar correspondence works for a higher order $n \geq 2$ by imposing the condition $h^{(m)}(a) = h^{(m)}(b) = 0$ ($0 \leq m < n/2$)
  and the periodic one $h^{(m)}(a) = h^{(m)}(b)$ if $m = n/2$. 
\end{Remark}

\begin{Lemma}\label{second}
  Choose $\delta>0$ small enough so that
  $x + iy$ ($x \in [a,b]$, $0 < y \leq \delta$) belongs to the domain of simple behavior of $\varphi$. 
 
  \begin{enumerate}
 \item 
 For $h \in C_0^2[a,b]$ and $0 < y \leq \delta$, 
  \begin{align*}
    \int_{a}^b h(x) \varphi(x + iy)\, dx
    &= \int_{a}^b h(x) \varphi(x+i\delta)\, dx + i(\delta - y) \int_{a}^b h'(x) \varphi(x+i\delta)\, dx\\
    &+  \int_y^\delta (t-y) \Bigl(h'(b) \varphi(b+it) - h'(a) \varphi(a+it) \Bigr)\, dt\\
     & - \int_{a}^b dx\,  h''(x) \int_y^\delta (t-y)  \varphi(x+it)\, dt. 
  \end{align*}
\item
  As a function of $y \in (0,\delta]$ and $x \in [a,b]$, $y\varphi(x+iy)$ is continuous and bounded so that
  the absolutely convergent improper integral 
  \[
    \int_0^\delta y \varphi(x+iy)\, dy = \lim_{\epsilon \downarrow 0} \int_\epsilon^\delta y \varphi(x+iy)\, dy
  \]
  is bounded and measurable as a function of $x \in [a,b]$.
\item
  For $h \in C_0^2[a,b]$, we have 
  \begin{align*}
    \lim_{y \to +0} \int_{a}^b h(x) \varphi(x + iy)\, dx
    &= \int_{a}^b h(x) \varphi(x+i\delta)\, dx
    + i\delta \int_{a}^b h'(x) \varphi(x+i\delta)\, dx\\
    &+ h'(b) \int_0^\delta t \varphi(b+it)\, dt - h'(a) \int_0^\delta t \varphi(a+it)\, dt\\
    &- \int_{a}^b dx\,  h''(x) \int_0^\delta t \varphi(x+it)\, dt. 
  \end{align*}   
\end{enumerate}
\end{Lemma}

\begin{proof}
%
(i) For $0 < y \leq \delta$, let $\varphi_y$ be a functional on $C_0^2[a,b]$ defined by
  \[
    \varphi_y(h) = \int_{a}^b h(x) \varphi(x+iy)\, dx.
  \]
  The fundamental formula in calculus of second order is then applied to $\varphi_y(h)$ to obtain
  \[
    \varphi_y(h) = \varphi_\delta(h) + (y-\delta) \left.\frac{d\varphi_t(h)}{dt}\right|_{t=\delta}
    + \int_\delta^y (y-t) \frac{d^2\varphi_t(h)}{dt^2}\, dt.
  \]
  The coefficient of the linear term is integrated-by-parts to become
  \begin{align*}
    \left.\frac{d\varphi_t(h)}{dt}\right|_{t=\delta}
    &= \int_{a}^b h(x) \left.\frac{\partial}{\partial t} \varphi(x + it)\right|_{t=\delta}\, dx
    = i \int_{a}^b h(x) \varphi'(x + i\delta)\, dx\\
    &= - i \int_{a}^b h'(x) \varphi(x + i\delta)\, dx,  
  \end{align*}
  whereas 
  \[
     \frac{d^2}{dt^2}\varphi_t(h)
    = \int_{a}^b h(x) \frac{\partial^2}{\partial t^2} \varphi(x + it)\, dx 
    = i^2 \int_{a}^b h(x) \varphi''(x + it)\, dx
  \]
  is integrated-by-parts twice to become 
  \[
    h'(b) \varphi(b+it) - h'(a) \varphi(a+it) - \int_{a}^b h''(x) \varphi(x+it)\, dx
  \]
  and hence
  \begin{multline*}
    \int_y^\delta (t-y)  \frac{d^2}{dt^2}\varphi_t(h)\, dt =
    \int_y^\delta (t-y) \Bigl(h'(b) \varphi(b+it) - h'(a) \varphi(a+it) \Bigr)\, dt\\
    - \int_y^\delta dt\, (t-y) \int_{a}^b h''(x) \varphi(x+it)\, dx. 
  \end{multline*}
  
 (ii) From the simple behavior assumption $t |\varphi(x+it)| \leq C$ for $x \in [a,b]$ and $0 < t \leq \delta$, one sees 
  \[
    (0 \vee (t-y)) |\varphi(x+it)|
    \leq C \frac{0 \vee(t-y)}{t} \leq C
  \]
  if $0 < y \leq \delta$ and $t>0$.
  In view of $\lim_{y \downarrow 0} 0 \vee (t-y) = t$ for each $0 < t \leq \delta$,
  we can apply the bounded convergence theorem to the function $(0\vee (t-y)) \varphi(x+it)$ of $t \in (0,\delta]$ for each $x \in [a,b]$
  to see that the continuous function 
  \[
    \int_y^\delta (t-y) \varphi(x+it)\, dt = \int_{(0,\delta]} (0 \vee(t-y)) \varphi(x + it)\, dt 
  \]
  of $x \in [a,b]$ converges pointwise to 
  \begin{align*}
    \lim_{y \downarrow 0} \int_y^\delta (t-y) \varphi(x+it)\, dt
    &= \int_{(0,\delta]} \lim_{y \downarrow 0} (0 \vee(t-y)) \varphi(x + it)\, dt\\
    &= \int_{(0,\delta]} t \varphi(x+it)\, dt,  
  \end{align*}
  which is a bounded measurable function of $x \in [a,b]$.
  
  (iii) The bounded convergence theorem is used to the above convergence again and we obtain
  \begin{align*}
    \lim_{y \downarrow 0} \int_y^\delta dt\, (t-y)
    &\int_{a}^b h''(x) \varphi(x+it)\, dx\\ 
    &= \int_{a}^b dx\, h''(x) \lim_{y \downarrow 0} \int_y^\delta (t-y) \varphi(x+it)\, dt\\
    &=  \int_{a}^b dx\, h''(x) \int_0^\delta t \varphi(x+it)\, dt, 
  \end{align*}
  which together with the convergence in (ii) proves the assertion. 
\end{proof}

\begin{Corollary}\label{C2}
  As a bounded linear functional of $C_0^2[a,b]$, the boundary limit $\displaystyle \lim_{y \to +0} \varphi(x+iy)\, dx$ exists and is given by
  \[
    \lim_{y \to +0} \int_a^b h(x) \varphi(x+iy)\, dx = \int_a^b h''(t) \Phi(t)\, dt,
  \]
  where
  \begin{multline*}
    \Phi(t) = \int_t^b (x + i\delta - t) \varphi(x+i\delta)\, dx
    - \frac{b-t}{b-a} \int_a^b (x + i\delta - a) \varphi(x+i\delta)\, dx\\
    + \frac{t-a}{b-a} \int_0^\delta y \varphi(b+iy)\, dy
    + \frac{b-t}{b-a} \int_0^\delta y \varphi(a+iy)\, dy\\
    - \int_0^\delta y \varphi(t+iy)\, dy
  \end{multline*}
  is a bounded measurable function of $t \in [a,b]$ and satisfies $\Phi(a) = \Phi(b) = 0$. 
\end{Corollary}

\begin{proof}
  Just express $h', h$  in the formula (iii) by $h''$ and then use the identity
  \[
  \int_a^b dx \int_a^x dt = \int_a^b dt \int_t^b dx 
  \]
  to reverse the order of repeated integrations. 
\end{proof}

\begin{Example}
  For $\varphi(z) = -1/z$ and $\delta>0$,
  \[
    \int_0^\delta y\varphi(x+iy)\, dy =
    \begin{cases}
      i\delta - x\log(x+i\delta) + x\log x &(x>0)\\
      i\delta &(x=0)\\
      i\delta - x\log(-x-i\delta) + x\log(-x) &(x < 0)
    \end{cases}
  \]
  and
  \begin{multline*}
    \Phi(x) = -x\log|x| + \frac{b-x}{b-a} a\log |a| + \frac{x-a}{b-a} b\log |b|\\
    - i\pi (0\wedge x)
    + i\pi (0\wedge a) \frac{b-x}{b-a}
    + i\pi (0\wedge b) \frac{x-a}{b-a}.   
  \end{multline*}
  These are continuous functions of $x \in \R$ but not differentiable at $x=0$.
  Notice also that $\Phi$ does not depend on the cutoff parameter $\delta>0$. 
\end{Example}

We now apply the above lemma to the function 
\[
h(x) = \left.\frac{\partial g_r(x)}{\partial r}\right|_{r=1} = i (1+a^2) \frac{xf(x)}{(1+x^2)^2}
    - i \frac{a^2-x^2}{2(1+x^2)} f'(x) 
\]
in Lemma~\ref{first} to conclude that (recall $\alpha = 2\arctan(a)$)
\[
(*)\quad  \lim_{r \uparrow 1} \int_{-\alpha}^\alpha f(\tan\frac{t}{2}) \phi(re^{it})\, dt
  = -i \lim_{y \to +0} \int_{-a}^a f(s) \varphi(s+iy) \frac{2}{1+s^2}\, ds
\]
if $f$ is an analytic function on $[-a,a]$ satisfying $f(\pm a) = 0$.

Since $\varphi(z)$ behaves simply on $[-a,a] + i0$, so does $\phi(e^{iz})$ on $[-\alpha,\alpha] + i0$ by Lemma~\ref{simples} and
we can talk about boundary limits
$\varphi(s +i0)$ on $(-a,a)$ and $\phi(e^{it - 0})$ on $(-\alpha,\alpha)$ as distributions.

The above equality $(*)$ then indicates the distributional relation
  \[
(**)\quad \int_\R f(\tan\frac{t}{2}) \phi(e^{it \mp 0})\, dt
  =-i  \int_\R f(s) \varphi(s \pm i0) \frac{2}{1+s^2}\, ds
\]
valid for $f \in C_c^\infty(-a,a)$.

In fact, by Corollary~\ref{C2}, both $\varphi(x+i0)$ on $(-a,a)$ and $\phi(e^{i\theta - 0})$ on $(-\alpha,\alpha)$ are extended to 
linear functionals on $C_0^2[-a,a]$ and on $C_0^2[-\alpha,\alpha]$ respectively.

Let $\sA_0[-a,a]$ be the linear subspace of $C_0^2[-a,a]$ consisting of analytic functions on $[-a,a]$ vanishing at $\pm a$ and similarly for
$\sA_0[-\alpha,\alpha]$. Since M\"obius transforms give rise to analytic change-of-variables,
the correspondence $f(s) \longleftrightarrow f(\tan(t/2))$ induces an isomorphism between $\sA_0[-a,a]$ and $\sA_0[-\alpha,\alpha]$.

Moreover, for an analytic function $F$ on $[-a,a]$, its normalized indefinite integral $f$ belongs to $\sA_0[-a,a]$.
Since the set of analytic functions on $[-a,a]$ is dense in $C[-a,a]$ relative to the uniform norm,
$\sA_0[-a,a]$ is dense in $C_0^2[-a,a]$ as well.

Thus, the equality $(*)$ is valid even for $f \in C_0^2[-a,a] \supset C_c^\infty(-a,a)$.
In particular, the distributional equality $(**)$ holds.

As a summary of discussions so far, we have proved 

\begin{Lemma}
  Let $\varphi$ be a holomorphic function defined on an open neighborhood of $[-a,a] \pm i0$ in $\C_\pm$
  and assume that $\varphi$ behaves simply on $[-a,a] \pm i0$.
  Then the equality $(*)$ holds for any $f \in C_0^2[-a,a]$ and $(**)$ for any $f \in C_c^\infty(-a,a)$.
\end{Lemma}

\begin{Corollary}\label{basic}
   Let $\varphi$ be a holomorphic function defined on an open neighborhood of $\R \pm i0$ in $\C_\pm$
  and assume that $\varphi$ behaves simply on $\R \pm i0$ so that the boundary limits $\varphi(s\pm i0)$ on $\R$ and
  $\phi(e^{it \mp 0})$ on $(-\pi,\pi)$ exist as distributions. Then we have
  \[
\int_\R f(\tan\frac{t}{2}) \phi(e^{it \mp 0})\, dt
  =-i  \int_\R f(s) \varphi(s \pm i0) \frac{2}{1+s^2}\, ds
\]
for $f \in C_c^\infty(\R)$. 
\end{Corollary}

This corollary identifies the most part $-\pi < t < \pi$ of the periodic distribution $\phi(e^{it\mp 0})$ on $t \in \R$ as
a distribution $\varphi(s \pm i0)$ on $s \in \R$.

To recover the exceptional point $\pi + 2\pi \Z$, we notice that a rotation of the $z$-plane by an angle $\theta$ is described in the $w$-plane by
\begin{align*}
  i \frac{1 - e^{i\theta}z}{1 + e^{i\theta}z}
  &= i \frac{i+w - e^{i\theta}(i+w)z}{i + w + e^{i\theta}(i+w)z}
    = i \frac{i+w - e^{i\theta}(i-w)}{i + w + e^{i\theta}(i-w)}\\
  &= i \frac{(e^{-i\theta/2} + e^{i\theta/2})w + i(e^{-i\theta/2} - e^{i\theta/2})}{(e^{-i\theta/2} - e^{i\theta/2})w + i(e^{-i\theta/2} + e^{i\theta/2})}\\
  &=
    \begin{pmatrix}
      \cos(\theta/2) & \sin(\theta/2)\\
      - \sin(\theta/2) & \cos(\theta/2)
    \end{pmatrix}. w, 
\end{align*}
i.e., the M\"obius transform by the rotation matrix of angle $-\theta/2$.
Here the dot indicates the matric action on $w \in \overline{\C}$ by the M\"obius transformation.
We shall return to this in a later section. 

For the choice $\theta = \pm \pi$, the involutive map $z \mapsto -z$ in the $z$-plane is then transferred to the inversion
$w \mapsto -1/w$ in the $w$-plane and we have
$\phi(-z) = -i \varphi(-1/w)$ so that
the distribution $\phi(e^{i(t+\pi) \mp 0})$ shifted by the half period $\pi$ is related to the inversion
$\widetilde\varphi(w) \equiv \varphi(-1/w)$ by
\[
  \lim_{r \to 1 \mp 0} \int_{-\beta}^\beta f(\tan\frac{t}{2}) \phi(re^{i(t+\pi)})\, dt =
  - i \lim_{y \to \pm 0} \int_{-b}^b f(s) \widetilde\varphi(s + iy)\frac{2}{1+s^2}\, ds
\]
for $b>0$ and $f \in C_0^2[-b,b]$ with $0 < \beta \equiv 2\arctan b < \pi$.

If we rewrite the left hand side as
\[
  \lim_{r \to 1\mp 0} \int_{\pi - \beta}^{\pi + \beta} f(-\cot t) \phi(re^{it})\, dt
\]
and notice the fact that $-\cot t$ is the inversion of $\tan t$,
the inversion of $\displaystyle \frac{2}{1+s^2} \widetilde\varphi(s\pm i0)$ of $-b < s < b$,
which is a distribution on $\overline{\R} \setminus [-1/b,1/b]$, turns out to describe $\phi(e^{it\mp 0})$ for $\pi-\beta < t < \pi + \beta$.

In particular, when $ab > 1$, on the overlapped region $(-a,a) \setminus [-1/b,1/b]$,
both $\frac{2}{1+s^2} \varphi(s \pm i0)$ and the inversion of $\frac{2}{1+s^2} \widetilde\varphi(s \pm i0)$
give rise to the same distribution $i\phi(e^{it\mp 0})$ on $(\beta-\pi,\pi-\beta) \setminus [-\alpha,\alpha]$, whence
$\frac{2}{1+s^2} \varphi(s \pm i0)$ coincides with the inversion of $\frac{2}{1+s^2} \widetilde\varphi(s \pm i0)$
on the region $(-a,a) \setminus [-1/b,1/b]$. 


Now rather long discussions are finally summarized as follows:

\begin{Theorem}
  Under the assumption of Corollary~\ref{basic}, the following hold:
  \begin{enumerate}
  \item
    The inversion $\widetilde\varphi(w) = \varphi(-1/w)$ shares the property with $\varphi$.
  \item
    Limits $\varphi(x\pm i0) = \lim_{y \to \pm 0} \varphi(x + iy)$ and
    $\widetilde\varphi(x \pm i0) = \lim_{y \to \pm 0} \widetilde\varphi(x + iy)$ exist as distributions on $\R$.
  \item
    On $\R^\times = \R \setminus \{ 0\}$, $\varphi(x \pm i0)/(1+x^2)$ and $\widetilde\varphi(x\pm i0)/(1+x^2)$ are related by
    the inversion: For $f \in C_c^\infty(\R^\times)$, we have 
    \[
      \int_{\R^\times} f(x) \frac{\varphi(x \pm i0)}{1+x^2}\, dx
      =       \int_{\R^\times} \widetilde{f}(x) \frac{\widetilde\varphi(x \pm i0)}{1+x^2}\, dx. 
    \]
  \item
    The joined distribution on $\overline{\R}$, ensured by (iii) and also denoted by $\varphi(x\pm i0)$,
    is related with the M\"obius transform of the boundary distribution $\phi(e^{it\mp 0})$ on $e^{it} \in \T$ by
    \[
      \oint_{2\pi} f(\tan\frac{t}{2}) \phi(e^{it\mp 0})\, dt
      = -i \int_{\overline{\R}} f(s) \varphi(s\pm i0) \frac{2}{1+s^2}\, ds
    \]
    for $f \in C^\infty(\overline{\R})$ ($f(-\infty) = f(\infty)$ particularly). 
  \end{enumerate}
\end{Theorem}

\section{Boundary Measures}
In this part, we consider a holomorphic function $\varphi$
defined on a boundary neighborhood of an open subset $S$ of $\overline{\R} = \R \sqcup \{\infty\}$,
which is assumed to behave simply on $S$ and therefore 
limits $\varphi(x\pm i0)$ exist as distributions on $S$.

We say that $\varphi$ admits a boundary measure on $S$ if
the distribution
\[
  \frac{1}{2\pi i(1+x^2)} \Bigl(\varphi(x+i0) - \varphi(x-i0) \Bigr)
\]
is given by a complex Radon measure $\lambda$  
(called the \textbf{boundary measure} of $\varphi$) on $S$, i.e.,
\[
  \lim_{y \to +0} \int_S f(s) \frac{\varphi(s+iy) - \varphi(s-iy)}{2\pi i(1+s^2)}\, ds
  = \int_S f(s)\, \lambda(ds) 
\]
for any $f \in C_c^\infty(S)$. 

When $\varphi$ admits a boundary measure $\lambda$,
let $\varphi_\lambda$ be a holomorphic function of $z \in \overline{\C} \setminus \overline{S} \supset \C \setminus \R$
($\overline{S}$ being the closure of $S$ in $\overline{\R}$)
defined by 
\[
  \varphi_\lambda(z) = \int_S \frac{1+sz}{s-z}\, \lambda(ds). 
\]
Then $\varphi_\lambda(\pm i) = \pm i \lambda(S)$ and $\lambda$ is a representing measure of $\varphi_\lambda$, 
whence $\varphi_\lambda$ behaves simply on $\overline{\R}$ and the inversion formula for $\lambda$ implies
\[
  \varphi(x+i0) - \varphi(x-i0) = \varphi_\lambda(x+i0) - \varphi_\lambda(x-i0)
\]
as distributions on $S$.
Since each term in this identity has a meaning as a distribution on $S$, we can rewrite it as 
\[
  \varphi(x+i0) -\varphi_\lambda(x+i0) = \varphi(x-i0) - \varphi_\lambda(x-i0)
\]
and apply the edge-of-the-wedge theorem (\cite[Theorem B]{Ru}) 
to see that $\psi \equiv \varphi - \varphi_\lambda$ is holomorphic on $S$ and satisfies 
$\psi(x+iy) - \psi(x-iy) \to 0$ ($y \to 0$) locally uniformly on $S$.
Consequently, the convergence 
\[
  \lim_{y \to +0} \frac{\varphi(x+iy) - \varphi(x-iy)}{2\pi i (1+x^2)} dx
  = \lim_{y \to +0} \frac{\varphi_\lambda(x+iy) - \varphi_\lambda(x-iy)}{2\pi i (1+x^2)} dx 
\]
has a meaning as a Radon measure on $S$ and results in the initial boundary measure $\lambda$, i.e.,
the boundary measure relation 
\[
  \lim_{y \to +0} \int_{-\infty}^\infty f(x) \frac{\varphi(x+iy) - \varphi(x-iy)}{2\pi i(1+x^2)} dx
  = \int_S f(x)\, \lambda(dx)
\]
remains valid for $f \in C_c(S)$. 


  

\begin{Theorem}\label{main}
  A holomorphic function $\varphi$ on $\C \setminus \R$ admits a representing measure if and only if the following conditions are satisfied.
  \begin{enumerate}
  \item 
   $\varphi$ behaves simply on $\overline{\R}$.
  \item
    We can find an open subset $S$ of $\R$ so that $\R \setminus S$ is a finite set
    and $\varphi$ admits a boundary measure on $S$, i.e., there is a complex Radon measure $\lambda_S$ in $S$ satisfying
    \[
      \lim_{y \to +0} \int_{-\infty}^\infty f(x) \frac{\varphi(x+iy) - \varphi(x-iy)}{2\pi i (1+x^2)} dx
      = \int_S f(x)\, \lambda_S(dx)
   \]
   for any $f \in C_c^\infty(S)$. 
    \end{enumerate}

    Moreover, if this is the case and if we define a holomorphic function $\varphi_S$ on $\C \setminus \R$ by
    \[
      \varphi_S(z) = \int_S \frac{1+sz}{s-z}\, \lambda_S(ds), 
    \]
    then $\varphi - \varphi_S$ is holomorphically extended to $(\C \setminus \R) \sqcup S$
    with each point in $\Sigma \equiv (\R \setminus S) \sqcup \{\infty\}$ being at most a simple pole of $\varphi - \varphi_S$, 
    the representing measure $\lambda$ of $\varphi$ coincides with $\lambda_S$ on $S$ and
    the atomic mass of $\lambda$ at $\sigma \in \Sigma$ is given by 
    \[
      \lambda[\sigma] = 
      \begin{cases}
        - \rho(\sigma)/(1+\sigma^2) &(\sigma \in \R)\\
        \rho(\sigma) &(\sigma = \infty)
      \end{cases}, 
    \]
  where $\rho(\sigma)$ is the residue of $\varphi - \varphi_S$ at $\sigma \in \Sigma$, i.e.,
  \[
    \varphi(z) - \varphi_S(z) - \frac{\rho(\sigma)}{z-\sigma}
  \]
  is holomorphic near $z = \sigma$ if $\sigma \in \R \setminus S$, whereas 
  \[
    \varphi(1/z) - \varphi_S(1/z) - \frac{\rho(\sigma)}{z}
  \]
  is holomorphic near $z=0$ if $\sigma = \infty$. 
\end{Theorem}

\begin{proof}
  Clearly the listed conditions are necessary.
  
  To see the reverse implication, we recall that $\varphi - \varphi_S$ is holomorphic on $S$, whence $\varphi - \varphi_S$
  has at most finitely many singularities in $\Sigma$ as a function on $\overline{\C}$.
  
  Since both $\varphi$ and $\varphi_S$ behaves simply on $\overline{\R}$,
  each point $\sigma \in \Sigma$ must be at most a simple pole of $\varphi - \varphi_S$ and
  \[
    \varphi(z) - \varphi_S(z) - \sum_{\sigma \in \R \setminus S} \frac{\rho(\sigma)}{z-\sigma} - \rho(\infty) z
  \]
  is constant as a holomorphic function on the whole sphere $\overline{\C}$.
  Thus, in terms of the measure $\lambda$ on $\overline{\R}$ described in the theorem, we obtain 
  \[
    \varphi(z) = \text{const} + \int_{\overline{\R}} \frac{1+sz}{s-z}\, \lambda(ds)
  \]
  and $\lambda$ is a representing measure of $\varphi$. 
\end{proof}

\begin{Remark}
  Under the existence of a representing measure, the convergence to the boundary measure $\lambda_S$ in (ii) holds for any
  continuous function $f$ on $S$ vanishing on the boundary of $S$ (see Appendix~C). 
\end{Remark}

\section{M\"obius Transform of Representing Measures}
In accordance with the previous article \cite{HY},
a holomorphic map of $\C_+$ into itself is called an \textbf{endofunction} on $\C_+$, 
though it is customary to be named after Carath\'eodory, Herglotz, Nevanlinna, Pick, Riesz and Schur.
An endofunction is obviously extended to a holomorphic function $\varphi$ on $\C \setminus \R$ so that $\varphi = \varphi^*$
(reflection-symmetric extension)
and the classic theorem of G.~Herglotz and F.~Riesz says that
the reflection-symmetric extension of an endofunction admits a positive representing measure
(simply referred to as the representing measure of an endofunction).
M\"obius transforms of positive representing measures are investigated in \cite{HY}.
Here we shall describe the results for complex representing measures.

Start with some notation. Given an invertible matrix
\[
  A =
  \begin{pmatrix}
    a & b\\
    c & d
  \end{pmatrix}
\]
of complex entries, the associated M\"obius transform of $z \in \overline{\C}$ is denoted by
\[
  A.z = \frac{az+b}{cz+d}. 
\]
An endofunction is called an \textbf{autofunction} if it maps $\C_+$ bijectively onto $\C_+$.
Let $\End(\C_+)$ be the set of endofunctions and $\Aut(\C_+)$ be the set of autofunctions.
Remark that $\End(\C_+)$ is a unital semigroup by composition and $\Aut(\C_+)$ is a subgroup of $\End(\C_+)$.
Compared to $\End(\C_+)$, $\Aut(\C_+)$ is very restrictive and in fact realized by the M\"obius transform of matrices in $\text{SL}(2,\R)$. 
Turning attention to $\C \setminus \R$ instead of $\C_+$, the M\"obius transform is extended to $\text{GL}(2,\R)$:
Half planes $\C_\pm$ are preserved individually or switched according to the sign of determinant.

Given a Radon measure $\lambda$ on $\overline{\R}$ and a matrix $A$ in $\text{GL}(2,\R)$,
let $\lambda^A$ be the measure on $\overline{\R}$ defined by
\[
  \lambda^A(dt) = \frac{1}{\det(A)} \frac{(at+b)^2 + (ct+d)^2}{1+t^2} \lambda(A.dt). 
\]
Here $\lambda(A.dt)$ denotes the transferred measure of $\lambda(ds)$ under the change of variables $s = A.t$.

Now we add the following formula for a record, which is an immediate modification of the one in \cite{HY}. 

\begin{Theorem}
  For $A \in \text{GL}(2,\R)$,
  a holomorphic function $\varphi(z)$ of $z \in \C \setminus \R$ admits a representing measure $\lambda$
  if and only if so does the holomorphic function $\varphi(A.z)$ of $z \in \C \setminus \R$.
  
  Moreover, if this is the case, the representing measure of $\varphi(A.z)$ is given by $\lambda^A$.

  Notice that, for a positive $\lambda$, $\lambda^A$ is positive or negative according to the sign of $\det(A)$. 
\end{Theorem}

\section{Examples}
In this section, we illustrate the results in the previous section by some of simple examples.
At the outset, recall the following fact.

\begin{Lemma}\label{mass}
  If a holomorphic function $\varphi$ on $\C \setminus \R$ admits a representing measure $\lambda$ on $\overline{\R}$,
  then 
  \[
    i(1+x^2)\lambda[x] = \lim_{y \to 0} y\varphi(x+iy)
  \]
  for $x \in \R$ and
  \[
    i\lambda[\infty] = \lim_{y \to \infty} \frac{\varphi(iy)}{y}.
  \]
  Here $\lambda[x]$ denotes the atomic mass $\lambda(\{x\})$ at $x \in \overline{\R}$. 
\end{Lemma}

\begin{proof}
  This is well-known for endofunctions (and their reflection-symmetric extensions) (cf.~\cite[Appendix A]{HY}, \cite[\S 1.12]{RR})
  and the general case follows from this special one as a linear combination of endofunctions. 

 In fact, from the formula 
  \begin{align*}
    \lim_{y \to +0} y \frac{\varphi(x+iy) + \varphi^*(x+iy)}{2} &= i(1+x^2) \frac{\lambda + \overline{\lambda}}{2}[x],\\
    \lim_{y \to +0} y \frac{\varphi(x+iy) - \varphi^*(x+iy)}{2i} &= i(1+x^2) \frac{\lambda - \overline{\lambda}}{2i}[x]
  \end{align*}
 for reflection-symmetric and antisymmetric parts of $\varphi$, one sees 
  \[
    \lim_{y \to +0} y\varphi(x+iy) = i(1+x^2) \lambda[x],
    \quad
     \lim_{y \to +0} y\varphi(x-iy) = -i(1+x^2) \lambda[x], 
 \]
   which are combined into the single formula
   \[
      \lim_{y \to 0} y\varphi(x+iy) = i(1+x^2) \lambda[x]. 
    \]

    Similarly for the atomic mass $\lambda[\infty]$ at $\infty$. 
     \end{proof}

 \begin{Example}\label{tangent}
   Recall that an endofunction $\varphi$ on $\C_+$ gives another endofunction $-1/\varphi$ on $\C_+$.
   Consequently $1/\varphi$ admits a representing measure, simply the negative of the representing measure for $-1/\varphi$.
   If this observation is applied to $\varphi(z) = \tan z$, we see that
   \[
     \tan z + \cot z = \frac{2}{\sin(2z)}
   \]
   has a representing measure.

   In fact, 
    \[
    \tan(x+iy) = \frac{1}{i} \frac{e^{-y+ix} - e^{y-ix}}{e^{-y+ix} + e^{y-ix}}
    = \frac{2\sin(2x) + i(e^{2y} - e^{-2y})}{(e^y - e^{-y})^2 + 4 \cos^2x}
  \]
  reveals that $\tan z$ is an endofunction.
  From $\tan i = i(e - e^{-1})/(e + e^{-1})$, the real constant for $\tan z$ is zero.
  The representing measure $\lambda$ of tangent function is then atomic in view of
  $[\varphi] = \{\pi n/2; n \in 1 + 2\Z\} \sqcup \{ \infty\}$. The atomic mass at $\infty$ is zero because of 
  $\lim_{y \to +\infty} (\tan(iy))/y = 0$, whereas
  $\lim_{z \to \pi n/2} (\pi n/2 - z)\tan z = 1$ shows that the atomic mass at $\pi n/2$ is $1/(1+(\pi n/2)^2)$.
  Consequently we have an absolutely convergent series expansion
  \[
    \tan z = \sum_{n \in 1 + 2\Z} \frac{1 + \pi nz/2}{\pi n/2 - z} \frac{1}{1 + (\pi n/2)^2}
  \]
  for $z \in \C \setminus (\pi/2)(1 + 2\Z)$.
  
  Similarly the boundary support of $\cot z$ is $(\pi \Z) \sqcup \{\infty\}$ and
  the atomic mass at $n\pi$ given by $-1/(1 + (\pi n)^2)$ with no atomic mass at $\infty$.

  Thus the boundary support of $1/\sin(2z)$ is $(\pi/2)\Z \sqcup \{\infty\}$ with the atomic mass at $\pi n/2$ equal to
  $(-1)^{n+1}/(1+ (\pi n/2)^2)$: We have an absolutely convergent expression 
  \[
    \frac{2}{\sin(2z)} = \sum_{n \in \Z} \frac{(-1)^{n+1}}{1 + (\pi n/2)^2} \frac{ 1+ \pi n z/2}{\pi n/2 -z}
  \]
  for $z \not\in (\pi/2)\Z$.

  By scaling and then shifting, both $1/\sin z$ and $1/\cos z$ have representing measures.
  As a conclusion, trigonometric functions admit representing measures other than entire ones, $\cos z$ and $\sin z$. 
 \end{Example}

 \begin{Example}[cf.~\cite{Do, Na}]
   Consider a rational function $\varphi$ which is holomorphic on $\C \setminus \R$, i.e., all poles of $\varphi$ are real
   (including $\infty$). If $\varphi$ has a representing measure, it behaves simply at any point of the boundary $\overline{\R}$,
   whence all poles of $\varphi$ must be simple. Then the partial fraction decomposition of $\varphi$ takes the form
   \[
     \varphi(z) = az+b + \sum_{j=1}^n \frac{c_j}{s_j - z}, 
   \]
   where $\{ s_j\}$ is the set of real poles (excluding $\infty$) and $a$, $b$, $c_j \not= 0$ are complex numbers.
   If we compare this with the representing measure relation 
   \[
     \varphi(z) = \frac{\varphi(i) + \varphi(-i)}{2} + \lambda[\infty] z + \sum_{j=1}^n \frac{1+s_jz}{s_j-z} \lambda[s_j], 
   \]
   we have
   \[
     c_j = (1+s_j^2) \lambda[s_j],\quad   a = \lambda[\infty],\quad
     b = \frac{\varphi(i) + \varphi(-i)}{2} -  \sum_j s_j \lambda[s_j], 
   \]
   which in turn determines $\lambda[s_j]$, $\lambda[\infty]$ and $(\varphi(i)+\varphi(-i))/2$ uniquely from
   $a$, $b$ and $(c_j)$ so that
   the obtained atomic measure $\lambda$ supported by $\{s_j\} \sqcup \{ \infty\}$ meets the relation of the representing measure. 

   In that case, $\deg \varphi \in \{ -1,0,1\}$ and
   \[
     \deg \varphi =
     \begin{cases}
       1 &(\lambda[\infty] \not= 0)\\
       0 &\text{($\lambda[\infty] = 0$ and $b \not= 0$)}\\
       -1 &\text{($\lambda[\infty] = 0$ and $b = 0$)}
     \end{cases}.
   \]
 \end{Example}

 In the following examples, the principal branch of logarithm is used and denoted by $\log z$:
 $\log(re^{i\theta}) = \log r + i\theta$ for $r>0$ and $-\pi < \theta <\pi$.
 For $p \in \C$, the $p$-th power function $z$ is then defined by $z^p = e^{p\log z}$ ($z \in \C \setminus (-\infty,0]$).
 The logarithm is real in the sense that $\log^* = \log$ and
 the power function satisfies $(z^p)^* = z^{\overline{p}}$, i.e., $\overline{z^p} = (\overline{z})^{\overline{p}}$. 
 For the reciprocal, $\log(1/z) = -\log z$ and $(1/z)^p = 1/z^p$.
 For the inversion $-1/z$, the boundary support is changed from $[-\infty,0]$ to $[0,\infty]$ so that 
 \[
   \log(-1/z) = \begin{cases}
                         i\pi - \log z &(0 < \theta < \pi)\\
                         -i\pi - \log z &(-\pi < \theta < 0)
                \end{cases}
              \]
 and 
 \[
  (-1/z)^p =
  \begin{cases}
    e^{i\pi p}/z^p &(0 < \theta < \pi)\\
    e^{-i\pi p}/z^p &(-\pi < \theta < 0)
  \end{cases}. 
 \]
 
 \begin{Example}\label{power}
   $z^p$ ($p \in \C \setminus \Z$).
   The boundary support is $[-\infty,0]$ and the power function $z^p$ behaves simply except for $\{ 0,\infty\}$.
   In view of
   \[
     yz^p = r^{1+p} e^{ip\theta} \sin\theta \quad(z = x + iy), 
   \]
   it behaves simply at $0$ if and only if $1+\text{Re}\,p \geq 0$, whereas
   \[
     y(-1/z)^p = e^{\pm i\pi p} r^{1-p} e^{-ip\theta}\sin\theta 
   \]
   according to $\pm y>0$ shows that the simple behavior holds at $\infty$ if and only if $1-\text{Re}\,p \geq 0$.

   In total, $z^p$ behaves simply on $\overline{\R}$ if and only if $-1 \leq \text{Re}\,p \leq 1$.

   In view of $|z^p| = |z|^{\text{Re}\,p} e^{-(\text{Im}\,p) \theta}$, $z^p$ is bounded near $z = 0$ for $\text{Re}\,p \geq 0$ and
   $|z^p| \leq |x|^{\text{Re}\,p} e^{\pi|\text{Im}\,p|}$ for $\text{Re}\, p < 0$, which are combined with
   \[
     \lim_{y \to +0} (x\pm iy)^p =
     \begin{cases}
       x^p &(x>0),\\
       |x|^p e^{\pm i\pi p} &(x<0)
     \end{cases}
   \]
   to see that, if $\text{Re}\,p > -1$, we have the Radon measure convergence
   \[
     \lim_{y \to +0} \int_\R f(x) (x\pm iy)^p\, dx = \int_0^\infty f(x) x^p\, dx + e^{\pm i\pi p}\int_{-\infty}^0 f(x) (-x)^p\, dx 
   \]
   for $f \in C_c(\R)$.
   
   Consequently, if $-1 < \text{Re}\,p < 1$, the measure convergence
   \[
     \lim_{y \to +0} \Bigl((x+iy)^p - (x-iy)^p\Bigr) dx = 2i (-x)^p \sin(\pi p)\, dx
   \]
   on $(-\infty,0)$ shows that the boundary measure exists on $\R^\times$,
   whence by Theorem~\ref{main} the representing measure of $z^p$ supported by $(-\infty,0)$ exists and is given by
   \[
     \frac{(-x)^p}{\pi(1+x^2)} \sin(\pi p)\, dx
     \quad
     (x < 0) 
   \]
   without atomic measures at $0$ and $\infty$. 
   Note that the representing measure is positive or negative according to $0 < p < 1$ or $-1 < p < 0$.

   For $p = 1 +iq$ with $0 \not= q \in \R$,
   \[
     \frac{e^{iy}}{y} = r^{iq} e^{-\pi q/2}
   \]
   circulates on the circle of radius $e^{-\pi q/2}$ and does not converge as $r \to \infty$,
   whence $\varphi$ has no representing measure.
   Similarly for $p = -1 + iq$ with $0 \not= q \in \R$.


   As a conclusion, the power function $z^p$ ($p \not\in \Z$) has a representing measure if and only if $-1 < \text{Re}\,p < 1$.
   Note that, for $p \in\Z$, $z^p$ admits a representing measure if and onif $p \in \{-1,0,1\}$. 
\end{Example}

\begin{Example}
  $z^p \log z$ ($p \in \C$). 
  The boundary support is $[-\infty,0]$ again and $z^p\log z$ behaves simply on $\overline{\R}$ if and only if $-1 < \text{Re}\,p < 1$.
  In that case, the representing measure supported by $(-\infty,0)$ exists without atomic mass and is given by
  \[
   \left( \frac{\sin\pi p}{\pi} \log|x| + \cos \pi p \right) \frac{|x|^p}{1+x^2}\, dx
 \]
 on $ (-\infty,0)$ with the following integral representation of $z^p\log z$:
 \[
   -\pi \sin\left(\frac{\pi}{2}p\right)
   + \int_{-\infty}^0 \frac{1+sz}{s-z}
   \left( \frac{\sin\pi p}{\pi} \log(-s) + \cos \pi p \right) \frac{(-s)^p}{1+s^2}\, ds.
\]

 Note that, if $-1 < p < 1$, the density factor
 \[
   \frac{\sin\pi p}{\pi} \log(-s) + \cos \pi p
 \]
 is monotone on $(-\infty,0)$ and changes its sign at $s = - e^{-\pi \cot(\pi p)}$ unless $p = 0$. 
\end{Example}

\begin{Example}
  $z^p/\log z$ ($p \in \C$). 
  The boundary support is $[-\infty,0] \sqcup \{1\}$ and
  $z^p/\log z$ behaves simply on $\overline{\R}$ if and only if $-1 \leq \text{Re}\,p \leq 1$.
  The function is meromorphic on $\C \setminus (-\infty,0]$ with $1$ unique pole and of multiplicity one. The mass at $1$ is then 
  \[
    \frac{1}{i(1+1)} \lim_{y \to +0} \frac{y(1 + iy)^p}{\log(1+iy)} = - \frac{1}{2}. 
  \]
  The function admits boundary limits on $\R \setminus \{ 0, 1, \infty\}$ so that 
  \[
    \lim_{y \to +0} \frac{(x\pm iy)^p}{\log(x\pm iy)}
    =
    \begin{cases}
      x^p/\log x &( 1 \not= x > 0),\\
      (-x)^pe^{\pm i\pi p}/(\log(-x) \pm i\pi) &(x < 0). 
    \end{cases}
  \]
  By a similar argument as in the case of power functions, this implies 
  \begin{multline*}
    \lim_{y \to +0} \left( \frac{(x\pm iy)^p}{\log(x\pm iy)} - \frac{(x\pm iy)^p}{\log(x\pm iy)} \right) dx\\
    = \Bigl( 2i\sin(\pi p) \log(-x) - 2\pi i \cos(\pi p) \Bigr) \frac{(-x)^p}{\pi^2 + (\log(-x))^2} dx
  \end{multline*}
  as a Radon measure on $(-\infty,0)$, which
  makes $1/(1+x^2)$ integrable if and only if $-1 < \text{Re}\, p < 1$.

  Conversely suppose that $p$ is in this strip region.
  Given $0 < \epsilon < 1$, we can find $M_\epsilon > 0$ so that
  \[
    \left| \frac{z^p}{\log z} \right|
    \leq M_\epsilon \frac{|x|^{\text{Re}\,p}}{-\log |x|}
  \]
  for $z = x+iy$ with $-1+\epsilon \leq x \leq 1-\epsilon$ and $0 < y \leq \epsilon)$.
  Since $|x|^{\text{Re}\, p}/\log|x|$ is integrable on $[-1+\epsilon,1-\epsilon]$ by the assumption,
  the dominated convergence theorem is applied to see that 
  \[
    \lim_{y \to +0} \frac{(x\pm iy)^p}{\log(x \pm iy)} dx =
   (0,1-\epsilon] \frac{x^p}{\log x} dx + [-1+\epsilon,0) e^{\pm i\pi p} \frac{(-x)^p}{\log(-x) \pm i\pi} dx 
  \]
  as a Radon measure on $[-1+\epsilon,1-\epsilon]$, whence
  \begin{multline*}
    \lim_{y \to +0} \left( \frac{(x\pm iy)^p}{\log(x\pm iy)} - \frac{(x\pm iy)^p}{\log(x\pm iy)} \right) dx\\
    = \Bigl( 2i\sin(\pi p) \log(-x) - 2\pi i \cos(\pi p) \Bigr) \frac{(-x)^p}{\pi^2 + (\log(-x))^2} dx - i \delta(x-1)
  \end{multline*}
  as a Radon measure on $\R$. 

  Thus $z^p/\log z$ admits a representing measure by Theorem~\ref{main}.

  As a conclusion, $z^p/\log z$ admits a representing measure if and only if $-1 < \text{Re}\, p < 1$ and,
  if this is the case, the boundary integral representation is given by 
  \begin{multline*}
    \frac{z^p}{\log z} = \frac{2}{\pi} \sin\left(\frac{\pi}{2} p\right)  - \frac{1}{2} \frac{1+z}{1-z}\\
    +  \int_{-\infty}^0
    \frac{1+sz}{s-z} \left(\frac{\sin(\pi p)}{\pi} \log(-s) - \cos(\pi p)\right)
    \frac{(-s)^p}{(1+s^2)(\pi^2 + (\log(-s))^2)}\, ds.  
  \end{multline*}
  Note that 
  \[
   \frac{1}{2} \left( \frac{i^p}{\log i} + \frac{(-i)^p}{\log(-i)} \right) = \frac{2}{\pi} \sin\left(\frac{\pi}{2}p\right). 
  \]
 \end{Example}

\begin{figure}[h]
\centering
\begin{minipage}[b]{0.49\columnwidth}
    \centering
    \includegraphics[width=0.6\columnwidth]{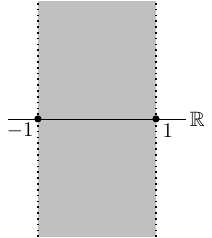}
    \caption{$p$ region of $z^p$}
\end{minipage}
\begin{minipage}[b]{0.49\columnwidth}
    \centering
    \includegraphics[width=0.6\columnwidth]{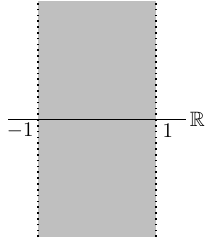}
    \caption{$p$ region of $z^p(\log z)^{\pm 1}$}
\end{minipage}
\end{figure}


\begin{Example}\label{powerbolic}
  As a combination of Example~\ref{tangent} and Example~\ref{power},
  given $\sigma > 0$, consider real (reflection-symmetric) meromorphic functions
  \[
    \tan( \sigma \log z) = -i \frac{z^{i\sigma} - z^{-i\sigma}}{z^{i\sigma} + z^{-i\sigma}},
    \quad
    - \cot(\sigma \log z) = -i \frac{e^{i\sigma} + z^{-i\sigma}}{z^{i\sigma} - z^{-i\sigma}}
  \]
  of $z \in \C \setminus (-\infty,0]$, which are (reflection-symmetric extensions of) endofunctions as compositions of endofunctions.
  In view of
  \[
    z^{i\sigma} = - z^{-i\sigma} \iff \log z \in \frac{\pi}{2\sigma} + \frac{\pi}{\sigma}\Z
  \]
  and
  \[
    z^{i\sigma} = z^{-i\sigma} \iff \log z \in \frac{\pi}{\sigma} \Z,
  \]
  these have simple poles at
  $e^{\pi(1+2k)/2\sigma}$ ($k \in \Z$) and $e^{\pi k/\sigma}$ ($k \in \Z$) respectively. Thus the meromorphic function 
  \[
    \varphi(z) = \frac{4i}{z^{2i\sigma} - z^{-2i\sigma}} = \frac{2}{\sin(2\sigma\log z)}
    = \tan(\sigma \log z) + \cot(\sigma \log z)
  \]
  of $z \in \C \setminus (-\infty,0]$ admits a representing measure as their difference. 
 
  Atomic masses of these functions are calculated by Lemma~\ref{mass}: Recalling the polar form expression 
  \[
  (x + iy)^{\pm i\sigma} = r^{\pm i\sigma} e^{\mp \sigma \theta} \quad (-\pi < \theta <\pi) 
  \]
  and the asymptotics
  \begin{align*}
    (e^{\pi n/2\sigma}+iy)^{\pm i\sigma}
    &= i^{\pm n} (1 + i e^{-\pi n/2\sigma} y)^{\pm i\sigma}\\
    &= i^{\pm n} \Bigl( 1 \mp \sigma e^{-\pi n/2\sigma}y \pm \frac{i\sigma(1 \mp i\sigma)}{2} e^{-\pi n/\sigma} y^2 + \cdots \Bigr)
  \end{align*}
  as $y \to 0$, we see 
 \begin{align*}
    \lim_{y \to 0} y \tan\Bigl(\sigma \log(e^{\pi n/2\sigma}+iy) \Bigr)
    &= \frac{i}{\sigma} e^{\pi n/2\sigma}
      \quad (n \in 1 + 2\Z),\\
       \lim_{y \to 0} y \cot\Bigl(\sigma \log(e^{\pi n/2\sigma}+iy) \Bigr)
    &= - \frac{i}{\sigma} e^{\pi n/2\sigma}
      \quad (n \in 2\Z)
 \end{align*}
 and  
 \[
   \lim_{y \to \infty} \frac{1}{y}\tan\Bigl(\sigma \log(iy) \Bigr)
   = \lim_{y \to \infty} \frac{1}{y}\cot\Bigl(\sigma \log(iy) \Bigr) = 0. 
 \]
 
 Thus there is no atomic mass at either $0$ or $\infty$ with the atomic mass at $e^{\pi n/2\sigma}$ ($n \in \Z$) given by 
 \[
   (-1)^{n+1} \frac{1}{\sigma} \frac{1}{e^{\pi n/2\sigma} + e^{-\pi n/2\sigma}}. 
 \]

  From expressions
 \begin{align*}
   \tan(\sigma \log z)
   &= \frac{4\cos(\sigma\log r) \sin(\sigma\log r) + i(e^{2\sigma\theta} - e^{-2\sigma\theta})}
     {(e^{\sigma\theta} + e^{-\sigma\theta})^2 \cos^2(\sigma\log r) + (e^{\sigma\theta} - e^{-\sigma\theta})^2 \sin^2(\sigma\log r)}\\
   \cot(\sigma \log z)
    &= \frac{4\cos(\sigma\log r) \sin(\sigma\log r) - i(e^{2\sigma\theta} - e^{-2\sigma\theta})}
     {(e^{\sigma\theta} - e^{-\sigma\theta})^2 \cos^2(\sigma\log r) + (e^{\sigma\theta} + e^{-\sigma\theta})^2 \sin^2(\sigma\log r)}
 \end{align*}
 in terms of the polar form $z = x + iy = re^{i\theta}$ ($r>0$, $-\pi < \theta < \pi$), their boundary limits on $(-\infty,0)$
 are obtained by letting $r = |x|$ and $\theta = \pm \pi$,
 whence their representing measures $\lambda_{\tan}$ and $\lambda_{\cot}$ take the following form
 \begin{align*}
 \pi(1+x^2) \lambda_{\tan}(dx) &= \frac{e^{2\pi \sigma} - e^{-2\pi\sigma}}
                      {(e^{\pi\sigma} + e^{-\pi\sigma})^2 \cos^2(\sigma\log |x|) + (e^{\pi\sigma} - e^{-\pi\sigma})^2 \sin^2(\sigma\log |x|)} dx\\
 \pi(1+x^2) \lambda_{\cot}(dx) &= \frac{e^{-2\pi\sigma} - e^{2\pi\sigma}}
     {(e^{\pi\sigma} - e^{-\pi\sigma})^2 \cos^2(\sigma\log |x|) + (e^{\pi\sigma} + e^{-\pi\sigma})^2 \sin^2(\sigma\log |x|)} dx
 \end{align*}
 on $(-\infty,0)$ and
 \begin{align*}
   \lambda_{\tan}[e^{\pi n/2\sigma}] &= \frac{1}{\sigma} \frac{1}{e^{\pi n/2\sigma} + e^{-\pi n/2\sigma}} (n \in 1 + 2\Z),\\
   \lambda_{\cot}[e^{\pi n/2\sigma}] &= - \frac{1}{\sigma} \frac{1}{e^{\pi n/2\sigma} + e^{-\pi n/2\sigma}} (n \in  2\Z). 
 \end{align*}
 Note that
 \[
   \tan(\sigma\log i) = i \tanh \frac{\pi\sigma}{2},
   \quad
   \cot(\sigma\log i) = -i \coth \frac{\pi\sigma}{2}
 \]
 and their reflection-symmetric constants vanish. 
\end{Example}
 
\appendix
\section{}
We reproduce here Vladimirov's estimate on the boundary behavior of holomorphic endofunctions on $\C_+$.
The original estimate is about functions of multiple variables and the single variable case below is therefore considerably simplified.

\begin{Theorem}[Vladimirov]\label{Vl}
 For an endofunction $\varphi$ on $\C_+$, 
  \[
    |\varphi(z)| \leq \frac{1+\sqrt{2}}{2} \frac{1+|z|^2}{\text{Im}\,z} |\varphi(i)|. 
  \]
\end{Theorem}

\begin{proof}
  Let $\phi(z)$ be the accompanied function on $D$ satisfying $\text{Re}\, \phi(z) > 0$ ($z \in D$).
  Then
  \[
   \psi = \frac{\phi - i \text{Im}\,\phi(0) - \text{Re}\,\phi(0)}{\phi - i \text{Im}\,\phi(0) + \text{Re}\,\phi(0)}
  \]
  is a normalized endofunction on $D$ and the Schwarz lemma gives $|\psi(z)| < |z|$ ($z \in D$).
  Thus
  $\phi - i\text{Im}\, \phi(0) = (\text{Re}\,\phi(0))(1+\psi)/(1-\psi)$ satisfies
  \[
    \text{Re}\,\phi(0) \frac{1 - |z|}{1 + |z|} \leq
    |\phi(z) - i \text{Im}\,\phi(0)|
    \leq \text{Re}\,\phi(0) \frac{1+|\psi(z)|}{1-|\psi(z)|}
    \leq \text{Re}\,\phi(0) \frac{1+|z|}{1-|z|}
  \]
  for $z \in D$ and hence 
  \[
    |\phi(z)| \leq |\text{Im}\,\phi(0)| + \text{Re}\,\phi(0) \frac{1+|z|}{1-|z|} 
    \leq |\phi(0)| \left(1 + \frac{1+|z|}{1-|z|}\right) = \frac{2|\phi(0)|}{1 - |z|}. 
  \]
  Consequently, in view of $\varphi(w) = i\phi(z)$ and $z = (i-w)/(i+w)$, we have
  \[
    | \varphi(w)| \leq \frac{2|\varphi(i)|}{1 - |\frac{i-w}{i+w}|}
  \]
  for $w \in \C_+$.

  The problem is therefore reduced to estimating the denominator. Let $w = u+iv$ ($u \in \R$, $v > 0$). 
  \[
    1 - \left| \frac{i-w}{i+w} \right| = \frac{|i+w|^2 - |i-w|^2}{|i+w|^2 + |i^2 - w^2|} = \frac{4v}{|i+w|^2 + |1+w^2|}.  
  \]
  In terms of the polar form $w = r e^{i\theta}$,
  \[
    |i+w|^2 + |1+w^2| = 1 + r^2 + 2r \sin\theta + \sqrt{(1+r^2)^2 - 4r^2 \sin^2\theta}
  \]
  Since
  \[
    F(t) = t + \sqrt{(r + r^{-1})^2/4 - t^2}
  \]
  for $0 \leq t \leq (r+r^{-1})/2$ is maximized at $t = (r + r^{-1})/2\sqrt{2}$ with the maximum
  \[
    F\left(\frac{r + r^{-1}}{2\sqrt{2}}\right) = \frac{r + r^{-1}}{\sqrt{2}}, 
  \]
  we have $|i+w|^2 + |1+w^2| \leq (1+\sqrt{2})(1+r^2)$.
%
\end{proof}

\begin{Remark}
 Vladimirov derived the estimate in \cite[Lemma~13.3.2]{Vl}, where the coefficient $\frac{1+\sqrt{2}}{2}$ is overestimated by $\sqrt{2}$. 
\end{Remark}

\begin{Corollary}
  If a holomorphic function $\varphi$ on $\C_+$ admits a representing measure, then
  \[
  \| \varphi\|_V \equiv \sup \{ \frac{\text{Im}\,z}{1+|z|^2} |\varphi(z)|; z \in \C_+\} < \infty.
  \]
\end{Corollary}

\section{}
For reader's convenience, we present a known fact that simple boundary behavior implies the existence of distributional boundary limits
according to \cite[Chap.~11]{BB}.

Consider a box region $B_\pm = (a,b) \pm i(0,\delta]$ in the upper/lower half plane $\C_\pm$, 
let $\varphi$ be a holomorphic function defined on an open neighborhood of $B_\pm$ in $\C_\pm$ 
and assume that $y^m\varphi(x\pm iy)$ ($a < x < b$, $0 < y \delta$) is bounded on $B_\pm$ for some integer $m \geq 0$.

We shall show that the boundary limit
\[
  \varphi(x\pm i0) = \lim_{y \to +0} \varphi(x \pm iy)
\]
exists as a distribution in $(a,b)$ and $\varphi(x\pm i0)$ ($a < x < b$) belongs to the dual of $C_0^{m+1}(a,b)$ in such a way that
\begin{multline*}
  \int_a^b f(x) \varphi(x\pm i0)\, dx = \sum_{k=0}^m \frac{(\pm i\delta)^k}{k!} \int_a^b f^{(k)}(x) \varphi(x\pm i\delta)\, dx\\
  + \frac{(\pm i)^{m+1}}{m!} \int_a^bdx\, f^{(m+1)}(x) \int_0^\delta t^m \varphi(x \pm it)\, dt.
\end{multline*}

By reflection symmetry, we may restrict ourselves to $B_+$. 
Let $f \in C_0^{m+1}(a,b)$, i.e., $f$ is a continuous function in the class $C^{m+1}$ on $(a,b)$ and
each $f^{(k)}$ satisfies $f^{(k)}(a+0) = 0 = f^{(k)}(b-0)$ for $0 \leq k \leq m+1$.  

Regard $\varphi(x+iy)$ ($0 < y < \delta$ with $\delta' > \delta$) as a linear functional $\varphi_y(f)$ of $f$:
\[
  \varphi_y(f) = \int_a^b f(x) \varphi(x+iy)\, dx.
\]
Clearly $\varphi_y(f)$ is a $C^\infty$ (in fact analytic) function of $0 < y < \delta'$ and
\begin{align*}
  \frac{d^k}{dy^k} \varphi_y(f)
  &= \int_a^b f(x) \frac{\partial^k}{\partial y^k} \varphi(x+iy)\, dx\\
  &= i^k \int_a^b f(x) \frac{\partial^k}{\partial x^k} \varphi(x+iy)\, dx\\
  &= (-i)^k \int_a^b f^{(k)}(x) \varphi(x+iy)\, dx
\end{align*}
for $0 \leq k \leq m+1$.

By the fundamental formula in calculus,
\begin{align*}
  \varphi_y(f)
  &= \sum_{k=0}^m \frac{(y-\delta)^k}{k!} \left.\frac{d^k}{dt^k} \varphi_t(f)\right|_{t=\delta}
    + \frac{1}{m!} \int_\delta^y (y-t)^m \frac{d^{m+1}}{dt^{m+1}} \varphi_t(f)\, dt\\
  &= \sum_{k=0}^m \frac{(y-\delta)^k}{k!} (-i)^k \int_a^b f^{(k)}(x) \varphi(x+i\delta)\, dx\\
  &\qquad\qquad+ \frac{i^{m+1}}{m!} \int^\delta_y dt\, (t-y)^m \int_a^b f^{(m+1)}(x) \varphi(x+it)\, dx
\end{align*}
and the limit $y \to +0$ question is narrowed down to the last integral.
Here the simple behavior of $\varphi$ comes into and
\[
  |t-y|^m |\varphi(x+it)|
  \leq C \frac{|t-y|^m}{t^m}  \leq C
 \]
 is bounded for $x \in (a,b)$ and $y \leq x$ in $(0,\delta]$.
 We now apply the bounded convergence theorem to 
 \begin{multline*}
   \int^\delta_y dt\, (t-y)^m \int_a^b f^{(m+1)}(x) \varphi(x+it)\, dx\\
   = \int^\delta_0 dt\, (0\vee (t-y))^m \int_a^b f^{(m+1)}(x) \varphi(x+it)\, dx
 \end{multline*}
 and obtain
 \begin{align*}
   \lim_{y \to +0} &\int^\delta_y dt\, (t-y)^m \int_a^b f^{(m+1)}(x) \varphi(x+it)\, dx\\
   &= \int^\delta_0 dt\, \lim_{y \to +0} \bigl(0\vee (t-y)\bigr)^m \int_a^b f^{(m+1)}(x) \varphi(x+it)\, dx\\
   &= \int^\delta_0 dt\, t^m \int_a^b f^{(m+1)}(x) \varphi(x+it)\, dx\\
   &= \int_a^bdx\, f^{(m+1)}(x) \int_0^\delta t^m \varphi(x+it)\, dt. 
 \end{align*}
 Note here that
 \[
   \int_0^\delta t^m \varphi(x+it)\, dt
   = \lim_{n \to \infty} \int_{1/n}^\delta t^m \varphi(x+it)\, dt
 \]
 is a bounded measurable function of $x \in (a,b)$ as a bounded sequential limit of continuous functons
 $\int_{1/n}^\delta t^m \varphi(x+it)\, dt$ ($n \geq 1$). 

 \section{}
 We shall describe the inversion formula for representing measures on the real line $\R$.
 Let $\lambda$ be the representing measure of a holomorphic function $\varphi$ on $\C \setminus \R$:
 \[
   \varphi(z) = \frac{\varphi(i) + \varphi(-i)}{2} + \int_{\overline{\R}} \frac{1+sz}{s-z}\, \lambda(ds).
 \]
 Then $\lambda$ on $\R$ is known to be recovered from $\varphi(z)$ by
 \[
   \lambda(dx) = \lim_{y \to +0} \frac{\varphi(x+iy) - \varphi(x-iy)}{2\pi i(1+x^2)}\, dx, 
 \]
 for which we clarify its meaning. 

 In the above integral representation,
 \[
     \frac{1+sz}{s-z} - \frac{1+s\overline{z}}{s - \overline{z}}
     = \frac{1+s^2}{s-z} - \frac{1+s^2}{s - \overline{z}} = (1+s^2) \frac{z - \overline{z}}{|s-z|^2}
     = \frac{2iy(1+s^2)}{(s-x)^2 + y^2}
 \]
 and hence
 \begin{align*}
   \varphi(x+iy) - \varphi(x-iy)
   &= 2iy \int_{\overline{\R}} \frac{1+s^2}{(s-x)^2 + y^2}\, \lambda(ds)\\
   &= 2i y \lambda[\infty] + 2iy \int_{\R} \frac{1+s^2}{(s-x)^2 + y^2}\, \lambda(ds)
 \end{align*}
 Thus the functional norm of the approximating measure is bounded by 
 \begin{multline*}
   \int_{-\infty}^\infty
   \frac{|\varphi(x+iy) - \varphi(x-iy)|}{1+x^2}\, dx\\
   \leq 2\pi y |\lambda[\infty]| + \int_{{\R}}|\lambda|(ds)\, (1+s^2) \int_{-\infty}^\infty \frac{2y}{(1+x^2)((s-x)^2 + y^2)}\, dx,
 \end{multline*}
 where the inner integral in the second term is residue-calculated as
 \begin{align*}
   \int_{-\infty}^\infty \frac{2y}{(1+x^2)((s-x)^2 + y^2)}\, dx
   &= \oint_{z=i} \frac{2y}{(1+z^2)((s-z)^2 + y^2)}\, dz\\
   &\qquad+ \oint_{z=s+iy} \frac{2y}{(1+z^2)((s-z)^2 + y^2)}\, dz\\
   &= \frac{2y}{2i((s-i)^2 + y^2)} \oint_{z=i} \frac{dz}{z-i}\\
   &\qquad+ \frac{2y}{(1+(s+iy)^2) 2iy} \oint_{z=s+iy} \frac{dz}{z-s-iy}\\
   &= 2\pi \left( \frac{y}{(s-i)^2 + y^2} + \frac{1}{1 + (s+iy)^2} \right). 
 \end{align*}
 The last expression is further reduced to
 \begin{align*}
  &\frac{y}{(s-i)^2 + y^2} + \frac{1}{1 + (s+iy)^2}\\
   &\qquad= \frac{1}{2i} \left( \frac{1}{s-i-iy} - \frac{1}{s-i+iy} \right)
     + \frac{1}{2i} \left( \frac{1}{s+iy-i} - \frac{1}{s+iy+i} \right)\\
   &\qquad= \frac{1}{2i} \left( \frac{1}{s-iy-i} - \frac{1}{s+iy+i} \right)
     = \frac{y+1}{s^2 + (y+1)^2}, 
 \end{align*}
 resulting in the identity 
 \[
   \int_{-\infty}^\infty \frac{2y}{(1+x^2)((s-x)^2 + y^2)}\, dx
   = \frac{2\pi(y+1)}{s^2 + (y+1)^2}. 
 \]
 The norm of the functional is therefore dominated by
 \[
   2\pi y|\lambda[\infty]| + 2\pi(y+1) \int_{{\R}} \frac{s^2+1}{s^2 + (y+1)^2} |\lambda|(ds)
   \leq 2\pi y|\lambda|(\overline{\R}) + 2\pi |\lambda|(\R). 
 \]
 
 When $\lambda \geq 0$, i.e., $(\varphi(x+iy) - \varphi(x-iy))/i \geq 0$,
 \begin{align*}
   &\int_{-\infty}^\infty  \frac{\varphi(x+iy) - \varphi(x-iy)}{i(1+x^2)}\, dx\\
   &\qquad= 2y \lambda[\infty] \int_{-\infty}^\infty \frac{1}{1+x^2}\, dx + \int_{-\infty}^\infty \frac{2y}{1+x^2} \int_\R \frac{1+s^2}{(s-x)^2 + y^2}\,
     \lambda(ds)\\
   &\qquad= 2\pi y \lambda[\infty] + 2\pi (y+1) \int_\R \frac{s^2+1}{s^2 + (y+1)^2}\, \lambda((ds).
 \end{align*}
 Then, by linearity on $\varphi$ and $\lambda$, the relation remains valid even for a complex measure $\lambda$. 

 As a summary, we have proved the following.

 \begin{Proposition}\label{preinversion}
   Let a $\varphi$ be a holomorphic function on the imaginary domain $\C \setminus \R$ and assume that $\varphi$ has
   a representing measure $\lambda$ on $\overline{\R}$.
   
   Then continuous functions $(\varphi(x+iy) - \varphi(x-iy))/(1+x^2)$ ($y > 0$) of $x \in \R$ are absolutely integrable
   and satisfy 
   \[
    \int_{-\infty}^\infty
    \frac{|\varphi(x+iy) - \varphi(x-iy)|}{1+x^2}\, dx
    \leq 2\pi y|\lambda|(\overline{\R}) + 2\pi |\lambda|(\R)
  \]
  and 
   \[
   \int_{-\infty}^\infty  \frac{\varphi(x+iy) - \varphi(x-iy)}{1+x^2}\, dx
   = 2\pi i \lambda[\infty] y +  2\pi i (y+1) \int_\R \frac{s^2+1}{s^2 + (y+1)^2}\, \lambda(ds).
 \]
 Here $|\lambda|$ is the total variation (measure) of $\lambda$. 
 Consequently,
 \[
   \lim_{y \to +0} \int_{-\infty}^\infty  \frac{\varphi(x+iy) - \varphi(x-iy)}{1+x^2}\, dx
   = 2\pi i \int_\R \frac{s^2+1}{s^2 + (y+1)^2}\, \lambda(ds).
 \]
 \end{Proposition}
 
 Let $f \in C(\overline{\R})$, i.e., $f$ is a bounded continuous function on $\R$ satisfying
 $\lim_{x \to +\infty} f(x) = \lim_{x \to -\infty} f(x)$, and consider a linear functional 
 \[
 \lambda_y(f) =  \int_{-\infty}^\infty f(x) \frac{\varphi(x+iy) - \varphi(x-iy)}{2\pi i(1+x^2)}\, dx 
 \]
 of $f \in C(\overline{\R})$. Notice that the integrand is continuous and absolutely integrable by Proposition~\ref{preinversion}.
 By the integral representation of $\varphi$, we have
 \[
   \lambda_y(f) = \frac{\lambda[\infty]}{\pi} y \int_{-\infty}^\infty \frac{f(x)}{1+x^2}\, dx
   + \frac{1}{\pi} \int_\R \lambda(ds) \int_{-\infty}^\infty f(x) \frac{1+s^2}{1+x^2} \frac{y}{(s-x)^2 + y^2}\, dx, 
 \]
 which is reduced to (ii) in Proposition~\ref{preinversion} for $f = 1$.
 
 In view of the formal relation
 \[
   \lim_{y \to +0} \frac{y}{(x-s)^2 + y^2} = \pi \delta(x-s), 
 \]
 the last term is expected to converge to
 \[
   \lambda_0(f) \equiv \int_\R f(s)\, \lambda(ds)
 \]
 as $y \to +0$.

 A precise reasoning is as follows:
 Given $\epsilon>0$, choose $L>0$ large enough to satisfy
 \[
 |\lambda|(\R \setminus [-L,L]) = |\lambda|(\R) -|\lambda|([-L,l]) \leq \epsilon. 
 \]
 Then
 \begin{align*}
   &\left| \int_{\R \setminus [-L,L]}
   \lambda(ds) \int_{-\infty}^\infty f(x) \frac{1+s^2}{1+x^2} \frac{y}{(x-s)^2 + y^2}\, dx \right|\\
   &\qquad\leq \pi (y+1) \| f\|_\infty \int_{\R \setminus [-L,L]} \frac{s^2+1}{s^2 + (y+1)^2}\, |\lambda|(ds)\\
   &\qquad\leq \pi (y+1) \| f\|_\infty\, |\lambda|(\R \setminus [-L,L]) \leq \pi (y+1) \| f\|_\infty \epsilon. 
 \end{align*}
 To handle the remaining part of the integral, recall that the convolution with the approximate delta functions
 \[
   \frac{1}{\pi} \frac{y}{x^2+y^2}
 \]
 of $x \in \R$ approximate the identity operator on $C_0(\R)$ as $y \to +0$, which enables us to 
 choose $\delta>0$ small enough to satisfy
 \[
   \left| \int_{-\infty}^\infty \frac{f(x)}{1+x^2} \frac{y}{\pi((s-x)^2 + y^2)}\, dx - \frac{f(s)}{1+s^2} \right| \leq \frac{\epsilon}{1+L^2}
 \]
 for $s \in \R$ and $0 < y < \delta$. Then
 \begin{multline*}
   \left| \int_{[-L,L]} \lambda(ds) \int_{-\infty}^\infty f(x) \frac{1+s^2}{\pi(1+x^2)} \frac{y}{(s-x)^2 + y^2}\, dx
     - \int_{[-L,L]} f(s)\, \lambda(ds)\right|\\
   \leq \epsilon \int_{[-L,L]} \frac{1+s^2}{1+L^2}\, |\lambda|(ds) \leq \epsilon |\lambda|([-L,L]) \leq \epsilon |\lambda|(\R). 
 \end{multline*}

 Putting all these together, we have
 \begin{multline*}
    \left| \int_{\R} \lambda(ds) \int_{-\infty}^\infty f(x) \frac{1+s^2}{\pi(1+x^2)} \frac{y}{(s-x)^2 + y^2}\, dx
      - \int_{\R} f(s)\, \lambda(ds)\right|\\
    \leq \| f\|_\infty (y + 2) \epsilon + |\lambda|(\R) \epsilon. 
  \end{multline*}
  Since $\epsilon > 0$ is arbitrary, we have proved the following.

  \begin{Theorem}\label{inversion}
    Under the same assumption in Proposition~\ref{preinversion}, we have for $f \in C(\overline{\R})$
    \[
      \lim_{y \to +0} \int_{-\infty}^\infty f(x) \frac{\varphi(x+iy) - \varphi(x-iy)}{2\pi i(1+x^2)}\, dx
      =  \int_{\R} f(s)\, \lambda(ds). 
    \]
    
    Notice that the point $\infty$ is missing in the right integral and therefore
    the atomic mass $\lambda[\infty]$ can not be recoverred in this form. 
  \end{Theorem}

\end{document}